\documentclass{amsart}
\usepackage{amsmath}
\usepackage{amssymb}
\usepackage{amsthm}
\usepackage{color}
\usepackage{caption}
\usepackage{graphicx, subfigure}
\usepackage{cite}
\usepackage{epstopdf}
\ifpdf
  \DeclareGraphicsExtensions{.eps,.pdf,.png,.jpg}
\else
  \DeclareGraphicsExtensions{.eps}
\fi

\title[ABCs for the solution of Heston model]{Artificial boundary method for the solution of pricing European options under the Heston model}
\author[H.S. Li, Z.Y. Huang] {Hongshan Li, Zhongyi Huang$^\dag$}

\address[H.S. Li, Z.Y. Huang]{Dept. of Mathematical Sciences, Tsinghua Univ., Beijing 100084, China}
\email{lhsusan1016@163.com; zhongyih@mail.tsinghua.edu.cn}
\thanks{This work was partially supported by NSFC Project No.~11871298.}
\thanks{$^\dag$ Corresponding author.}

\date{\today}

\begin{document}

\maketitle

\begin{abstract}
	This paper considers the valuation of a European call option under the Heston stochastic volatility model. We present the asymptotic solution to the option pricing problem in powers of the volatility of variance. Then we introduce the artificial boundary method for solving the problem on a truncated domain, and derive several artificial boundary conditions (ABCs) on the artificial boundary of the bounded computational domain. A typical finite difference scheme and quadrature rule are used for the numerical solution of the reduced problem. Numerical experiments show that the proposed ABCs are able to improve the accuracy of the results and have a significant advantage over the widely-used boundary conditions by Heston in the original paper \cite{Heston1993}.
	
%\centering
\end{abstract}

%--%--%--%--%--%--%--%--%--%--%--%--%--%--%--%--%--%
\section{Introduction}

Derivatives play an important role in financial markets and the valuation of derivatives has drawn widespread attention. A derivative is a financial security whose value depends on the price of one or more underlying assets. One example is the European call option which gives the holder the right to buy the underlying asset at a price fixed in advance on the expiration date. Option valuation is mainly based on intuition before Black, Scholes \cite{BlackScholes} and Merton \cite{Merton1973} made a major breakthrough in 1973. In their work, the price of the underlying asset follows a geometric Brownian motion, then a diffusion equation is derived by It\^{o}'s theory which results in an initial boundary value problem for the European option. The solution to the problem known as the Black-Scholes formula has been widely used in industry.

However, due to the strong assumption of constant drift and volatility, the Black-Scholes-Merton (BSM) model fails to capture the skewness and kurtosis of the asset prices and is not able to capture the volatility smile observed from the real market data. Numerous works have been done to develop models with better behaviors, including two main categories: the local volatility and stochastic volatility models. Dupire \cite{dupire1994pricing}, Derman and Kani \cite{derman1994riding} proposed the local volatility model by considering the volatility as a deterministic function that depends on both the stock price and time. It is easy to calibrate to the market data but not capable of capturing the dynamics of the smile by Hagan et al. \cite{hagan2002managing}.
Thus, the stochastic volatility models are usually preferred to the local volatility model in practice. Different stochastic processes have been adopted to describe the volatility and the most commonly used are the Ornstein-Uhlenbeck (OU) process in the Stein-Stein model \cite{stein1991stock} and the Cox-Ingersoll-Ross (CIR) process in the Heston model \cite{Heston1993}. Both of them have good performances on reflecting the features of the real market and have closed-form solutions for the European option. One advantage of the Heston model is that the CIR process can guarantee that the volatility is nonnegative while the volatility may go negative in the OU process in the Stein-Stein model. The Heston model has gained much attention among researchers and practitioners and become one of the most popular methods.

The closed-form solution of the option pricing problem for the European options under the Heston model involves in complex integrals which is not straightforward to calculate. As a result, some researchers have tried to perform asymptotic analysis to get a better understanding of the solution. Fouque et al. \cite{fouque2000derivatives} constructed a framework to study the asymptotic behavior of stochastic volatility models. Han et al. \cite{Asym2HanJiguang} presented an asymptotic expansion of the option price for a fast mean-reverting process in the Heston model and then Zhang et al. \cite{zhang2013option} pointed out that for a slow mean-reverting process the asymptotic solution has the same expression as in \cite{Asym2HanJiguang}.

However, the asymptotic solution can be applied only when the parameters in the Heston model such as the variance of volatility and the rate of mean reversion are small or large enough, and a lot of calculation may be needed to meet the accuracy requirements.
Besides, the closed-form or asymptotic solutions are difficult to derive for some exotic options or options with early exercise features. In order to give a consistent approach to the problem which is independent of the parameters and can be extended to a wider range of options, a number of numerical methods have been investigated to overcome the limitations.

Similar to the option valuation problem under the BSM model,  lattice approaches \cite{finucane1996american}\cite{leisen2000stock}\cite{vellekoop2009tree}, Monte Carlo simulation methods \cite{broadie2006exact}\cite{smith2007almost}\cite{andersen2008simple} and the FFT-based methods \cite{carr1999option}\cite{fang2008novel} have been developed to price options under the Heston model. Besides, since the Heston model results in a 2D convection-diffusion equation with a mixed 2nd-order derivative, finite difference methods and finite element methods have been adopted and improved, such as the Crank-Nicolson scheme \cite{zvan2000pde}, the alternating direction implicit schemes \cite{foulon2010adi}, the method of lines \cite{chiarella2012evaluation},  a mixture of standard Galerkin finite element method and finite volume method \cite{zvan1998penalty}
and the spectral element method \cite{zhu2010spectral}. In comparison, finite difference methods are more interpretable and flexible and are easy to implement requiring fewer pretreatment with lower computational cost.

It is noteworthy that the corresponding initial boundary value problem is defined on an infinite domain that needs to be handled carefully. In computational mathematics, alternative approaches for solving partial differential equations on unbounded domains include domain truncation\cite{engquist1977absorbing}\cite{antoine2007review}, coordinate transforms \cite{grosch1977numerical}, infinite element methods \cite{bettess1977infinite} and spectral methods\cite{boyd2001chebyshev}\cite{shen2009some}. Among these, the mainstream method is to truncate the domain and introduce an artificial boundary condition. The method was first developed by Engquist and Majda \cite{engquist1977absorbing} in 1977 to obtain an absorbing boundary condition for the wave equation. Since then, it has been applied to a number of equations in physics, finance and other areas \cite{han2002class}\cite{wong2008artificial}. As for the option valuation problem, Han and Wu \cite{HanWu} proposed a method to give an ABC for the BSM model. Unfortunately, it is difficult to impose a proper ABC for the problem under the Heston model because of its complexity. What is commonly used is the boundary condition at infinity by Heston in the original paper \cite{Heston1993}. Obviously, it could cause high inaccuracy.

 In this paper, we present an asymptotic solution in terms of the volatility of variance in the Heston model and it can be viewed as a generalized form of the results of Han et al. \cite{Asym2HanJiguang} and Zhang et al. \cite{zhang2013option} since they need assumptions about another parameter in addition. We obtain some properties of the option price as well based on asymptotic analysis. As for the numerical method, we choose the Samarskii scheme and the central difference scheme for the two directions to achieve 2nd-order accuracy. We use the backward Euler scheme for the first time step considering the nonsmoothness of the initial condition and then switch to the Crank-Nicolson scheme for the subsequent steps. Numerical results show that the finite difference scheme performs well enough to solve the initial boundary value problem for option pricing. To deal with the infinite domain, we try to give an approximation of the original problem with the asymptotic properties of the option price and derive an approximate ABC for the Heston model based on the approach of Han and Wu \cite{HanWu}. Then we make modifications and get another two ABCs to improve the accuracy. The Heston problem involves in an additional dimension and cause the violation of the compact support of the initial value and the source term which is often required in the artificial boundary methods, so that the modified artificial boundary conditions are not straightforward to derive. We make proper assumptions to change the problem into a common form and apply curve fitting approaches to solve the difficulties. Numerical experiments turn out that the proposed ABCs work remarkably better compared with the boundary condition by Heston in the original paper. For different sets of parameters including both large and small volatility of variance, the proposed ABCs can help to achieve considerably higher accuracy especially when the grid size gets smaller. With the original boundary condition by Heston, a much larger computational domain is needed to make the error of the solution close to those using the proposed ABCs. Besides, when we calculate the Greeks such as Delta, Gamma and Vega which are widely used in industry for trading and hedging, the proposed ABCs can reduce the error significantly and show the big advantage of the artificial boundary method in practice.

The paper is organized as follows. Section \ref{sec:Heston} introduces the Heston stochastic volatility model for option pricing and the initial boundary value problem for the European option. Section \ref{sec:asy} presents the asymptotic solution in powers of the volatility of variance in the Heston model and several useful properties of the option price based on asymptotic analysis. In Section \ref{sec:abc}, we construct an approximate artificial boundary condition and give two modified ABCs for the option valuation problem. Section \ref{sec:fdm} gives the finite difference scheme for the problem especially for the ABCs proposed in the previous section. Numerical results show the advantages of our ABCs compared with the boundary condition by Heston in Section \ref{sec:num}. We conclude in Section \ref{sec:con}.

%--%--%--%--%--%--%--%--%--%--%--%--%--%--%--%--%--%

\section{The Heston stochastic volatility model}\label{sec:Heston}

Under the Heston stochastic volatility model, the stock price $S_t$ is characterized by the following stochastic differential equation (SDE) under the risk neural probability measure:
\begin{equation}
dS_t=rS_tdt + \sqrt{v_t}S_tdz_{1,t},
\label{eq:stockSDE}
\end{equation}
where $\sqrt{v_t}$ represents the volatility governed by
\begin{equation}
dv_t=\kappa\left(\eta-v_t\right)dt+\sigma\sqrt{v_t}dz_{2,t}.
\label{eq:volSDE}
\end{equation}
Here, $z_{1,t}$, $z_{2,t}$ are two standard Brownian motions satisfying $E\left(dz_{1,t}dz_{2,t}\right)=\rho dt$. The parameter $r$ is the risk free rate of interest, $\kappa$ is the mean reversion rate, $\eta$ is the mean variance and $\sigma$ is the volatility of variance.

Denote the value of a financial derivative as $U(S,v,t)$, then by It\^{o}'s Lemma we arrive at the Garman's PDE:
\begin{equation}
\frac{1}{2}vS^2\frac{\partial^2U}{\partial S^2}+\rho\sigma vS\frac{\partial^2U}{\partial S\partial v}+\frac{1}{2}\sigma^2v\frac{\partial^2U}{\partial v^2}+rS\frac{\partial U}{\partial S}+\left[\kappa\left(\eta-v\right)-\lambda(S,v,t)\right]\frac{\partial U}{\partial v}
-rU+\frac{\partial U}{\partial t}=0.
\label{eq:HestonPDE}
\end{equation}
Here, $\lambda(S,v,t)$ is the price of volatility risk and in this paper we simply set it to $0$ without loss of generality.

Consider a European option with a strike price $K$ and a maturity time $T$, then $U(S,v,t)$ satisfies \eqref{eq:HestonPDE} on the domain
\begin{equation}
0<S<+\infty,\quad 0<v<+\infty,\quad 0\leq t<T,
\label{eq:HestonDomain}
\end{equation}
with the terminal condition
\begin{equation}
U(S,v,T)=\left(S-K\right)^+,\quad 0\leq S<+\infty, \quad 0\leq v<+\infty,
\label{eq:HestonIC}
\end{equation}
and the boundary conditions by Heston \cite{Heston1993},
\begin{equation}
\begin{aligned}
U(0,v,t) =0,\quad 0\leq v<+\infty, \quad 0\leq t<T,\\
\frac{\partial U}{\partial S}(S,v,t) \to1 \quad as \quad S\to +\infty,\quad 0\leq v<+\infty, \quad 0\leq t<T,\\
rS\frac{\partial U}{\partial S}(S,0,t)+\kappa\eta\frac{\partial U}{\partial v}(S,0,t)-rU(S,0,t)+\frac{\partial U}{\partial t}(S,0,t)=0,
\quad 0\leq S<+\infty, \quad 0\leq t<T,\\
U(S,v,t)\to S \quad as \quad v \to +\infty,\quad 0\leq S<+\infty, \quad 0\leq t<T.
\end{aligned}
\label{eq:HestonBCs}
\end{equation}
Assuming that the Feller condition,
\[\frac{1}{2}\kappa\eta\geq\sigma^2,\]
is satisfied, the volatility $\sqrt{v_t}$ never reaches zero and we are not supposed to put any boundary condition on $v=0$ according to the Fichera theory. In fact, no boundary condition is needed on $S=0$ either. However, the boundary conditions on both $S=0$ and $v=0$ in \eqref{eq:HestonBCs} come naturally from the original equation \eqref{eq:HestonPDE} and the terminal condition \eqref{eq:HestonIC}.

Let
\begin{equation}
\tau=T-t,\quad \tilde{S}=\frac{Se^{r\tau}}{K},\quad V(\tilde{S},v,\tau)=\frac{U(S,v,t)e^{r\tau}}{K},
\label{eq:HestonTrans}
\end{equation}
then the problem \eqref{eq:HestonPDE}-\eqref{eq:HestonBCs} is equivalent to
\begin{eqnarray}
\frac{1}{2}v\tilde{S}^2\frac{\partial^2V}{\partial \tilde{S}^2}+\rho\sigma v\tilde{S}\frac{\partial^2V}{\partial \tilde{S}\partial v}+\frac{1}{2}\sigma^2v\frac{\partial^2V}{\partial v^2}+\kappa\left(\eta-v\right)\frac{\partial V}{\partial v}=\frac{\partial V}{\partial \tau},\nonumber\\
0<\tilde{S}<+\infty,\quad 0<v<+\infty,\quad 0<\tau\leq T,
\label{eq:HestonProfirst}\\
V(\tilde{S},v,0)=(\tilde{S}-1)^+,\quad 0\leq\tilde{S}<+\infty,\quad 0\leq v<+\infty,
\label{eq:HestonProsecond}\\
V\left(0,v,\tau\right)=0,\quad 0\leq v<+\infty,\quad 0<\tau\leq T,\\
\frac{\partial V}{\partial \tilde{S}}(\tilde{S},v,\tau) \to1 \quad as \quad \tilde{S}\to +\infty,\quad 0\leq v<+\infty,\quad 0<\tau\leq T,\\
\kappa\eta\frac{\partial V}{\partial v}(\tilde{S},0,\tau)=\frac{\partial V}{\partial \tau}(\tilde{S},0,\tau),\quad 0\leq \tilde{S}<+\infty,\quad 0<\tau\leq T,\\
V(\tilde{S},v,\tau)\to \tilde{S} \quad as \quad v \to +\infty,\quad 0\leq \tilde{S}<+\infty,\quad 0<\tau\leq T.
\label{eq:HestonProlast}
\end{eqnarray}

%--%--%--%--%--%--%--%--%--%--%--%--%--%--%--%--%--%
\section{Asymptotic analysis for the option pricing problem}\label{sec:asy}

Considering $\sigma \rightarrow 0^+$, we find the asymptotic solution to the problem \eqref{eq:HestonProfirst}-\eqref{eq:HestonProlast} presented in the following theorem:
\newtheorem{theorem}{Theorem}[section]
\begin{theorem}
\label{thm1}
\label{thmAsymp}
The solution $V(\tilde{S},v,\tau)$ to the problem \eqref{eq:HestonProfirst}-\eqref{eq:HestonProlast} has the following asymptotic expansion in powers of $\sigma$:
\begin{equation}
V(\tilde{S},v,\tau)=V_0(\tilde{S},v,\tau)+\sigma V_1(\tilde{S},v,\tau)+\sigma^2 V_2(\tilde{S},v,\tau)+O(\sigma^3),\quad \sigma\rightarrow 0^+,
\label{eq:HestonAsymall}
\end{equation}
where
\begin{equation}
V_0(\tilde{S},v,\tau)=\tilde{S}\cdot N(d^+)-N(d^-),
\label{eq:HestonAsym0}
\end{equation}
\begin{equation}
V_1(\tilde{S},v,\tau)=-\frac{\rho}{2\kappa}F(v,\tau;\kappa,\eta)z^{-1}d^{-}\phi(d^{-}),
\label{eq:HestonAsym1}
\end{equation}
\begin{equation}
\begin{aligned}
V_2(\tilde{S},v,\tau)=&\frac{1}{4\kappa^2}\bigg\{G(v,\tau;\kappa,\eta)z^{-1}d^{-}+G(v,\tau;\kappa,\eta)z^{-\frac{3}{2}}\left[-1+(d^-)^2\right]+\rho^2J(v,\tau;\kappa,\eta)z^{-\frac{3}{2}}\left[-1+(d^-)^2\right]\\
&+\rho^2H(v,\tau;\kappa,\eta)z^{-2}\left[-3d^-+(d^{-})^{3}\right]+\rho^2H(v,\tau;\kappa,\eta)z^{-\frac{5}{2}}\left[3-6(d^-)^2+(d^-)^4\right]\bigg\}.
\end{aligned}
\label{eq:HestonAsym2}
\end{equation}
Here, $N(x)$ and $\phi(x)$ are the cumulative distribution function and the density of the standard normal distribution respectively and
\[d^{\pm}=\frac{\ln\tilde{S}\pm\frac{1}{2}z}{\sqrt{z}},\quad z=\eta\tau+\left(v-\eta\right)\frac{1-e^{-\kappa\tau}}{\kappa},\]
\[F(v,\tau;\kappa,\eta)=
\left(1-\frac{\kappa e^{-\kappa\tau}}{1-e^{-\kappa\tau}}\tau\right)z+\eta\left(\frac{\kappa e^{-\kappa\tau}}{1-e^{-\kappa\tau}}\tau^2-\frac{1-e^{-\kappa\tau}}{\kappa}\right),\]
\[
G(v,\tau;\kappa,\eta)=-\left(\frac{\kappa e^{-\kappa\tau}}{1-e^{-\kappa\tau}}\tau+\frac{1-e^{-\kappa\tau}}{2}-1\right)z
+\eta\left[\frac{\kappa e^{-\kappa\tau}}{1-e^{-\kappa\tau}}\tau^2-\frac{e^{-\kappa\tau}}{2}\tau-\frac{\left(1-e^{-\kappa\tau}\right)^2+2\left(1-e^{-\kappa\tau}\right)}{4\kappa}\right],
\]
\[H(v,\tau;\kappa,\eta)=\frac{1}{2}\left[F(v,\tau;\kappa,\eta)\right]^2,\]
\[
J(v,\tau;\kappa,\eta)=\left(-\frac{\kappa^2e^{-\kappa\tau}}{1-e^{-\kappa\tau}}\tau^2-2\frac{\kappa e^{-\kappa\tau}}{1-e^{-\kappa\tau}}\tau+2\right)z
+\eta\left(\frac{\kappa^2 e^{-\kappa\tau}}{1-e^{-\kappa\tau}}\tau^3+2\frac{\kappa e^{-\kappa\tau}}{1-e^{-\kappa\tau}}\tau^2+2e^{-\kappa\tau}\tau-4\frac{1-e^{-\kappa\tau}}{\kappa}\right).
\]
\end{theorem}

\begin{proof}
We first substitute the asymptotic expansion of $V(\tilde{S},v,\tau)$ w.r.t $\sigma$ in \eqref{eq:HestonAsymall} into the PDE \eqref{eq:HestonProfirst} and the initial condition \eqref{eq:HestonProsecond}. Since the coefficients of different powers of $\sigma$ should be zero, we get the following equations for $V_0(\tilde{S},v,\tau)$, $V_1(\tilde{S},v,\tau)$ and $V_2(\tilde{S},v,\tau)$:
\begin{eqnarray}
\label{eq:HestonAsymPDE0}
\frac{1}{2}v\tilde{S}^2\frac{\partial^2V_0}{\partial \tilde{S}^2}+\kappa\left(\eta-v\right)\frac{\partial V_0}{\partial v}=\frac{\partial V_0}{\partial \tau},\quad 0<\tilde{S}<+\infty,\quad 0<v<+\infty,\quad 0<\tau\leq T,\nonumber\\
V_0(\tilde{S},v,0)=(\tilde{S}-1)^+,\quad 0\leq\tilde{S}<+\infty,\quad 0\leq v<+\infty,\\
\label{eq:HestonAsymPDE1}
\frac{1}{2}v\tilde{S}^2\frac{\partial^2V_1}{\partial \tilde{S}^2}+\rho v\tilde{S}\frac{\partial^2V_0}{\partial \tilde{S}\partial v}+\kappa\left(\eta-v\right)\frac{\partial V_1}{\partial v}=\frac{\partial V_1}{\partial \tau},\quad 0<\tilde{S}<+\infty,\quad 0<v<+\infty,\quad 0<\tau\leq T,\nonumber\\
V_1(\tilde{S},v,0)=0,\quad 0\leq\tilde{S}<+\infty,\quad 0\leq v<+\infty,\\
\label{eq:HestonAsymPDE2}
\frac{1}{2}v\tilde{S}^2\frac{\partial^2V_2}{\partial \tilde{S}^2}+\rho v\tilde{S}\frac{\partial^2V_1}{\partial \tilde{S}\partial v}+\frac{1}{2}v\frac{\partial^2V_0}{\partial v^2}+\kappa\left(\eta-v\right)\frac{\partial V_2}{\partial v}=\frac{\partial V_2}{\partial \tau},\quad 0<\tilde{S}<+\infty,\quad 0<v<+\infty,\quad 0<\tau\leq T,\nonumber\\
V_2(\tilde{S},v,0)=0,\quad 0\leq\tilde{S}<+\infty,\quad 0\leq v<+\infty.
\end{eqnarray}
To solve the problem \eqref{eq:HestonAsymPDE0}, we introduce the following change of variables:
\[x=\ln \tilde{S},\quad \xi=(v-\eta)e^{-\kappa\tau}, \quad \gamma=\frac{1}{2}\left[\eta\tau+\xi\frac{e^{\kappa\tau}-1}{\kappa}\right]=\frac{1}{2}\left[\eta\tau+\left(v-\eta\right)\frac{1-e^{-\kappa\tau}}{\kappa}\right]:=\frac{1}{2}z,\]
\[W_0(x,\xi,\gamma)=V_0(\tilde{S},v,\tau)\cdot e^{-\frac{1}{2}x+\frac{1}{4}\gamma}.\]
Then the problem \eqref{eq:HestonAsymPDE0} is equivalent to
\begin{equation}
\begin{aligned}
\frac{\partial^2W_0}{\partial x^2}=\frac{\partial W_0}{\partial \gamma},\quad -\infty<x<+\infty,\quad 0<\eta+\xi e^{\kappa\tau}<+\infty,\quad 0<\tau\leq T,\\
W_0(x,\xi,\gamma)=(e^{\frac{1}{2}x}-e^{-\frac{1}{2}x})^+,\quad -\infty<x<+\infty,\quad 0\leq \eta+\xi e^{\kappa\tau}<+\infty.\\
\end{aligned}
\label{eq:HestonAsym0simp}
\end{equation}
The solution to the problem \eqref{eq:HestonAsym0simp} is given by
\begin{equation}
\tilde{W_0}(x,\xi,\gamma)=\frac{1}{\sqrt{4\pi\gamma}}\int_{-\infty}^{+\infty}e^{-\frac{(x-y)^2}{4\gamma}}(e^{\frac{1}{2}y}-e^{-\frac{1}{2}y})^+dy=e^{\frac{1}{4}\gamma}\left[e^{\frac{x}{2}}N\left(\frac{x+\gamma}{\sqrt{2\gamma}}\right)-e^{-\frac{x}{2}}N\left(\frac{x-
\gamma}{\sqrt{2\gamma}}\right)\right].
\label{eq:HestonAsym0comp}
\end{equation}
By transforming the variables back, we can derive the expression of $V_0(x,v,\tau)$ in \eqref{eq:HestonAsym0}.

Next, to solve the problem \eqref{eq:HestonAsymPDE1}, we substitute \eqref{eq:HestonAsym0} into it and let
\[x=\ln \tilde{S},\quad \xi=(v-\eta)e^{-\kappa\tau},\]
\[V_1(x,v,\tau)=W_1(v,\tau)\cdot d^-\phi(d^-),\]
then \eqref{eq:HestonAsymPDE1} can be transformed into
\begin{equation}
\begin{aligned}
\frac{\partial}{\partial \tau}\left[\left(\eta\tau+\xi\frac{e^{\kappa\tau}-1}{\kappa}\right)W_1\right]+(\eta+\xi e^{\kappa\tau})\rho\frac{1-e^{-\kappa\tau}}{2\kappa}=0,\quad
-\eta<\xi<+\infty,\quad 0<\tau\leq T,\\
W_1(\xi,0)=0,\quad -\eta\leq \xi<+\infty,
\label{eq:HestonAsym1inside2}
\end{aligned}
\end{equation}
with the solution
\begin{equation}
W_1(\xi,\tau)=-\frac{\rho}{2\kappa}z^{-1}\left[(\eta-\xi)\tau+\xi\frac{e^{\kappa\tau}-1}{\kappa}-\eta\frac{1-e^{-\kappa\tau}}{\kappa}\right]=-\frac{\rho}{2\kappa}z^{-1}F(v,\tau;\kappa,\eta).
\end{equation}
Then we can easily get the solution to the problem \eqref{eq:HestonAsymPDE1} in \eqref{eq:HestonAsym1}.

Besides, the expression of $V_2(\tilde{S},v,\tau)$ in \eqref{eq:HestonAsym2} can be derived similarly as above.
\end{proof}
\newtheorem{remark}{Remark}[section]
\begin{remark}
If we replace $\kappa$ and $\sigma$ by $\delta\kappa$ and $\sqrt{\delta}\sigma$ in Theorem \ref{thm1} with $\delta\rightarrow0^+$, we can get the asymptotic solution consistent with the results in \cite{Asym2HanJiguang} and \cite{zhang2013option}.
\end{remark}

We make the inverse transformation of \eqref{eq:HestonTrans} on $V_0(\tilde{S},v,\tau)$ in \eqref{eq:HestonAsym0} and denote $U_0(S,v,t)$ the derived expression:
\begin{equation}
U_0(S,v,t)=S\cdot N(d^+)-Ke^{-r(T-t)}N(d^-),\\
\label{eq:HestonAsym0U}
\end{equation}
where
\[d^\pm=\frac{\ln{\frac{S}{K}}\pm\frac{1}{2}z}{\sqrt{z}},\quad z=\eta(T-t)+(v-\eta)\frac{1-e^{-\kappa(T-t)}}{\kappa}.\]
According to Theorem \ref{thm1}, the option price $U(S,v,t)$ converges to $U_0(S,v,t)$ as $\sigma\rightarrow0^+$.

We take into account an asset $S_t^{(0)}$ with the dynamic specified as in \eqref{eq:stockSDE} and the variance $v_t^{(0)}$ as in \eqref{eq:volSDE} but with $\sigma$ set to 0, i.e.
\[dv_t^{(0)}=\kappa(\eta-v_t^{(0)})dt,\]
or
\begin{equation}
v_t^{(0)}=\eta+(v_0^{(0)}-\eta)e^{-\kappa t},
\end{equation}
where $v_0^{(0)}$ is a constant that represents the starting point of the variance $v_t^{(0)}$. Note that $v_t^{(0)}$ is here a function of $t$ but not a stochastic process as before. And
\begin{equation}
dS_t^{(0)}=rS_t^{(0)}dt+\sqrt{\eta+(v_0^{(0)}-\eta)e^{-\kappa t}}S_t^{(0)}dz_t^{(0)}.
\label{eq:HestonSDE0}
\end{equation}
Here, $z_t^{(0)}$ is a standard Brownian motion.

Consider a European option $U^{(0)}$ with the underlying asset $S_t^{(0)}$
and the same strike price and maturity time as $U(S,v,t)$, then $U^{(0)}$ should only be related to $S$ and $t$ but not $v$. Then we have the following theorem:
\begin{theorem}
 $U^{(0)}(S,t)=U_0(S,\eta+(v_0^{(0)}-\eta)e^{-\kappa t},t)$.
\end{theorem}
\begin{proof}
Based on \eqref{eq:HestonSDE0} and applying It\^{o}'s Lemma, we arrive at
\begin{equation}
\frac{1}{2}S^2\frac{\partial^2U^{(0)}}{\partial S^2}\left[\eta+(v_0-\eta)e^{-\kappa t}\right]+rS\frac{\partial U^{(0)}}{\partial S}-rU^{(0)}+\frac{\partial U^{(0)}}{\partial t}=0,\quad
0<S<+\infty,\quad 0\leq t<T,
\label{eq:HestonPDE0}
\end{equation}
and the initial and boundary conditions are
\begin{equation}
\begin{aligned}
U^{(0)}(S,T)=\left(S-K\right)^+,\quad 0\leq S<+\infty,\\
U^{(0)}(0,t) =0,\quad 0\leq t<T,\\
\frac{\partial U^{(0)}}{\partial S}(S,v,t) \to1 \quad as \quad S\to +\infty, \quad 0\leq t<T.\\
\end{aligned}
\label{eq:HestonPDEC0}
\end{equation}
Let
\[\tau=T-t,\quad \tilde{S}=\frac{Se^{r\tau}}{K},\quad V^{(0)}(\tilde{S},\tau)=\frac{U^{(0)}(S,t)e^{r\tau}}{K},\]
the problem \eqref{eq:HestonPDE0}-\eqref{eq:HestonPDEC0} is equivalent to
\begin{eqnarray}
\label{eq:HestonPDE0simp}
\frac{1}{2}\tilde{S}^2\frac{\partial^2V^{(0)}}{\partial \tilde{S}^2}\left[\eta+(v_0-\eta)e^{-\kappa(T-\tau)}\right]=\frac{\partial V^{(0)}}{\partial \tau},\quad 0<\tilde{S}<+\infty,\quad 0<\tau\leq T,\\
V^{(0)}(\tilde{S},0)=(\tilde{S}-1)^+,\quad 0\leq \tilde{S}<+\infty,\\
V^{(0)}(0,\tau) =0,\quad 0<\tau\leq T,\\
\frac{\partial V^{(0)}}{\partial \tilde{S}}(\tilde{S},\tau) \to1 \quad as \quad \tilde{S}\to +\infty, \quad 0<\tau\leq T.
\end{eqnarray}

Let
\[x=\ln\tilde{S},\quad \gamma^{(0)}=\frac{1}{2}\left[\eta\tau+(v_0^{(0)}-\eta)\frac{e^{-\kappa T}(e^{\kappa\tau}-1)}{\kappa}\right],\]
\[W^{(0)}(x,\gamma^{(a)})=V^{(0)}(\tilde{S},\tau)\cdot e^{-\frac{1}{2}x+\frac{1}{4}\gamma^{(0)}},\]
then we have
\begin{equation}
\begin{aligned}
\frac{\partial^2W^{(0)}}{\partial x^2}=\frac{\partial W^{(0)}}{\partial \gamma^{(0)}},\quad
-\infty<x<+\infty,\quad 0<\gamma^{(0)}\leq \frac{1}{2}\left[\eta T+(v_0-\eta)\frac{1-e^{-\kappa T}}{\kappa}\right],\\
W^{(0)}(x,0)=\left(e^{\frac{x}{2}}-e^{-\frac{x}{2}}\right)^+,\quad -\infty<x<+\infty.
\end{aligned}
\end{equation}
Thus, similarly to \eqref{eq:HestonAsym0comp}, we have
\begin{equation}
W^{(0)}(x, \gamma^{(0)})=\frac{1}{\sqrt{4\pi\gamma^{(0)}}}\int_{-\infty}^{+\infty}e^{-\frac{(x-y)^2}{4\gamma^{(0)}}}(e^{\frac{y}{2}}-e^{-\frac{y}{2}})^+dy\nonumber\\
=e^{\frac{1}{4}\gamma^{(0)}}\left[e^{\frac{x}{2}}N\left(\frac{x+\gamma^{(0)}}{\sqrt{2\gamma^{(0)}}}\right)-e^{-\frac{x}{2}}N\left(\frac{x-
\gamma^{(0)}}{\sqrt{2\gamma^{(0)}}}\right)\right],
\end{equation}
and
\begin{equation}
V^{(0)}(\tilde{S}, \tau)=\tilde{S}\cdot N\left(\frac{\ln{\tilde{S}}+\gamma^{(0)}}{\sqrt{2\gamma^{(0)}}}\right)-N\left(\frac{\ln{\tilde{S}}-
\gamma^{(0)}}{\sqrt{2\gamma^{(0)}}}\right).
\end{equation}
Then
\begin{equation}
U^{(0)}(S, t)=S\cdot N\left(\frac{\ln{\frac{S}{K}}+\gamma^{(0)}}{\sqrt{2\gamma^{(0)}}}\right)-Ke^{-r(T-t)}\cdot N\left(\frac{\ln{\frac{S}{K}}-
\gamma^{(0)}}{\sqrt{2\gamma^{(0)}}}\right),
\end{equation}
where
\[\gamma^{(0)}=\frac{1}{2}\left[\eta(T-t)+(v_0^{(0)}-\eta)\frac{e^{-\kappa t}-e^{-\kappa T}}{\kappa}\right].\]
Compared with \eqref{eq:HestonAsym0U}, it is obvious that
\[U^{(0)}(S,t)=U_0(S,\eta+(v_0^{(0)}-\eta)e^{-\kappa t},t)\]
\end{proof}
We study the asymptotic behaviors of the derivatives of $V(\tilde{S},v,\tau)$ and have the following theorem:
\begin{theorem}
\label{thmAsym3}
The following statements of the derivatives of $V$ hold true.
\begin{enumerate}
\item $ v^m\frac{\partial V}{\partial v}(\tilde{S}, v, \tau)\to 0\quad as\quad \sigma\to0^+,\quad v\to +\infty,\quad \forall m\in \mathbb{N},\quad 0\leq \tilde{S}<+\infty, \quad 0\leq\tau\leq T$;\vspace{1mm}
\item $\tilde{S}^m\frac{\partial^2 V}{\partial \tilde{S}^2}(\tilde{S}, v, \tau)\to 0 \quad as \quad\sigma\to0^+,\quad v\to +\infty,\quad \forall m\in \mathbb{N},\quad 0\leq \tilde{S}<+\infty, \quad  \tilde{S} \neq 1, \quad 0\leq\tau\leq T$;\vspace{1mm}
\item $ v^m\frac{\partial V}{\partial v}(\tilde{S}, v, \tau)\to 0\quad as\quad \sigma\to0^+,\quad \tilde{S}\to +\infty,\quad \forall m\in \mathbb{N},\quad 0\leq v<+\infty, \quad 0\leq\tau\leq T$;\vspace{1mm}
\item $\tilde{S}^m\frac{\partial^2 V}{\partial \tilde{S}^2}(\tilde{S}, v, \tau)\to 0 \quad as \quad\sigma\to0^+,\quad \tilde{S}\to +\infty,\quad \forall m\in \mathbb{N},\quad 0\leq v<+\infty, \quad 0\leq\tau\leq T$;\vspace{1mm}
\item $\frac{\partial V}{\partial \tau}(\tilde{S}, v, \tau)\to 0\quad as\quad \sigma\to0^+,\quad \tau\to 0,\quad \forall 0\leq\tilde{S}<+\infty,\quad  \tilde{S} \neq 1,\quad 0\leq v<+\infty$.
\end{enumerate}
\end{theorem}
\begin{proof}
First, we give the proof of the first four limits.

Since
\[V(\tilde{S},v,0)=(\tilde{S}-1)^+,\quad \frac{\partial V}{\partial v}(\tilde{S}, v, 0)=0, \quad \forall 0\leq \tilde{S}<+\infty,\quad 0\leq v<+\infty,\]
\[\frac{\partial^2 V}{\partial \tilde{S}^2}(\tilde{S}, v, 0)=0,\quad \forall 0\leq \tilde{S}<+\infty,\quad  \tilde{S} \neq 1,\quad 0\leq v<+\infty,
\]
the limits can be verified immediately.

Next we can only focus on $0<\tau\leq T$ and the asymptotic behaviors of the derivatives of $V_0(\tilde{S},v,\tau)$ and then easily extend them to those of $V(\tilde{S},v,\tau)$ as $\sigma\to 0^+$.

With \eqref{eq:HestonAsym0} we have
\begin{equation}
\frac{\partial V_0}{\partial v}(\tilde{S}, v, \tau)=\phi(d^-)\frac{1}{2\sqrt{z}}\frac{1-e^{-\kappa\tau}}{\kappa}=\frac{1}{2\sqrt{2\pi z}}\frac{1-e^{-\kappa\tau}}{\kappa}\sqrt{\tilde{S}}e^{-\frac{z}{8}}e^{-\frac{\left(\ln{\tilde{S}}\right)^2}{2z}},\\
\end{equation}
\begin{equation}
\frac{\partial^2 V_0}{\partial \tilde{S}^2}(\tilde{S}, v, \tau)=\phi(d^+)\frac{1}{\tilde{S}\sqrt{z}}=\frac{1}{\sqrt{2\pi z}}\frac{1}{\tilde{S}\sqrt{\tilde{S}}}e^{-\frac{z}{8}}e^{-\frac{\left(\ln{\tilde{S}}\right)^2}{2z}}.
\end{equation}
Note that for $0<\tau\leq T,\quad 0\leq v<+\infty$,
\[z=\eta\tau+(v-\eta)\frac{1-e^{-\kappa\tau}}{\kappa}\geq \eta\frac{e^{-\kappa\tau}-1+\kappa\tau}{\kappa}>0.\]
Then $\forall m\in\mathbb{N},\quad 0<\tau\leq T,\quad 0\leq \tilde{S}<+\infty$,
\begin{equation}
\lim_{v\to +\infty}v^m\frac{\partial V_0}{\partial v}(\tilde{S}, v, \tau)=\frac{1}{2\sqrt{2\pi}}\frac{1-e^{-\kappa\tau}}{\kappa}\sqrt{\tilde{S}}\lim_{v\to +\infty}v^m\frac{1}{\sqrt{z}}e^{-\frac{z}{8}}e^{-\frac{\left(\ln{\tilde{S}}\right)^2}{2z}}=0,
\end{equation}
\begin{equation}
\lim_{v\to +\infty}\tilde{S}^m\frac{\partial^2 V_0}{\partial \tilde{S}^2}(\tilde{S}, v, \tau)=\frac{1}{\sqrt{2\pi}}\tilde{S}^{m-\frac{3}{2}}\lim_{v\to +\infty}\frac{1}{\sqrt{z}}e^{-\frac{z}{8}}e^{-\frac{\left(\ln{\tilde{S}}\right)^2}{2z}}=0,
\end{equation}
and $\forall m\in\mathbb{N},\quad 0<\tau\leq T,\quad 0\leq v<+\infty$,
\begin{equation}
\lim_{\tilde{S}\to +\infty}v^m\frac{\partial V_0}{\partial v}(\tilde{S}, v, \tau)=\frac{1}{2\sqrt{2\pi z}}\frac{1-e^{-\kappa\tau}}{\kappa}v^m e^{-\frac{z}{8}}\lim_{\tilde{S}\to +\infty}\sqrt{\tilde{S}}e^{-\frac{\left(\ln{\tilde{S}}\right)^2}{2z}}=0,
\end{equation}
\begin{equation}
\lim_{\tilde{S}\to +\infty}\tilde{S}^m\frac{\partial^2 V_0}{\partial \tilde{S}^2}(\tilde{S}, v, \tau)=\frac{1}{\sqrt{2\pi z}}e^{-\frac{z}{8}}\lim_{\tilde{S}\to +\infty}\tilde{S}^{m-\frac{3}{2}}e^{-\frac{\left(\ln{\tilde{S}}\right)^2}{2z}}=0.
\end{equation}

For the last limit, similarly, we focus on the asymptotic behavior of $\frac{\partial V_0}{\partial \tau}(\tilde{S}, v, \tau)$ and extend it to $\frac{\partial V}{\partial \tau}(\tilde{S}, v, \tau)$ as $\sigma\to0^+$.
With
\eqref{eq:HestonAsym0} we have
\begin{equation}
\frac{\partial V_0}{\partial \tau}(\tilde{S}, v, \tau)=\phi(d^-)\frac{1}{2\sqrt{z}}\left(\eta+(v-\eta)e^{-\kappa\tau}\right)=\frac{\sqrt{\tilde{S}}}{2\sqrt{2\pi}}\frac{1}{\sqrt{z}}e^{-\frac{\left(\ln{\tilde{S}}\right)^2}{2z}-\frac{z}{8}}\left(\eta+(v-\eta)e^{-\kappa\tau}\right).
\end{equation}
When $\tilde{S}=0$,
\[V_0(0, v, \tau)=0, \quad \frac{\partial V_0}{\partial \tau}(0, v, \tau)=0,\quad \forall 0\leq v<+\infty,  \quad 0\leq\tau\leq T.\]
$\forall 0<\tilde{S}<+\infty, \quad \tilde{S} \neq 1, \quad 0\leq v<+\infty$,
\begin{equation}
\lim_{\tau\to 0}\frac{\partial V_0}{\partial \tau}(\tilde{S}, v, \tau)=\frac{\sqrt{\tilde{S}}}{2\sqrt{2\pi}}\lim_{\tau\to 0}\frac{1}{\sqrt{z}}e^{-\frac{\left(\ln{\tilde{S}}\right)^2}{2z}-\frac{z}{8}}\left(\eta+(v-\eta)e^{-\kappa\tau}\right)=0.
\end{equation}
\end{proof}

%--%--%--%--%--%--%--%--%--%--%--%--%--%--%--%--%--%

\section{Artificial boundary conditions for the option pricing problem}\label{sec:abc}
Choose $\tilde{S}_{max}>1$, we introduce an artificial boundary
 and the computational domain as follows:
\[\Gamma_{\tilde{S}}=\{(\tilde{S}, v, \tau)\mid \tilde{S}=\tilde{S}_{max},\quad 0\leq v<+\infty, \quad 0\leq \tau\leq T\},\]
\[\Omega_{\tilde{S}}=\{(\tilde{S}, v, \tau)\mid \tilde{S}_{max}<\tilde{S}<+\infty,\quad 0\leq v<+\infty, \quad 0\leq \tau\leq T\}.\]
Then we try to construct an artificial boundary condition on $\Gamma_{\tilde{S}}$ by restricting the problem \eqref{eq:HestonProfirst}-\eqref{eq:HestonProlast} on $\Omega_{\tilde{S}}$.

We consider the problem when $\sigma\to0^+$ first. By Theorem \ref{thmAsym3}, we know that
\[\kappa(\eta-v)\frac{\partial V}{\partial v}(\tilde{S},v,\tau)\to 0\quad as \quad \sigma\to0^+,\quad \tilde{S}\to+\infty.\]
So it is rational to omit all the three derivatives w.r.t. $v$ in the PDE \eqref{eq:HestonProfirst} for simplicity. Then, to reduce the error, we treat those three terms as a source term related to only $v$ and $\tau$ and then to $\tilde{S}$, $v$ and $\tau$.
%---------------------------------
\subsection{An approximate artificial boundary condition (ApABC)}
 First, we drop the three derivatives w.r.t. $v$ and get to the approximate problem as follows:
\begin{equation}
\begin{aligned}
\frac{1}{2}v\tilde{S}^2\frac{\partial^2V^{(a)}}{\partial \tilde{S}^2}=\frac{\partial V^{(a)}}{\partial\tau},\quad
\tilde{S}_{max}<\tilde{S}<+\infty,\quad 0<v<+\infty,\quad 0<\tau\leq T,\\
V^{(a)}(\tilde{S},v,0)=\tilde{S}-1,\quad \tilde{S}_{max}\leq\tilde{S}<+\infty,\quad 0\leq v<+\infty.
\label{eq:PDEforABC}
\end{aligned}
\end{equation}
We introduce some changes of variables and derive an artificial boundary condition based on the method by Han and Wu \cite{HanWu}.
Let
\[W^{(a)}(\tilde{S},v,\tau)=V^{(a)}(\tilde{S},v,\tau)-(\tilde{S}-1),\]
then we have
\begin{equation}
\begin{aligned}
\frac{1}{2}v\tilde{S}^2\frac{\partial^2W^{(a)}}{\partial \tilde{S}^2}=\frac{\partial W^{(a)}}{\partial\tau},\quad
\tilde{S}_{max}<\tilde{S}<+\infty,\quad 0<v<+\infty,\quad 0<\tau\leq T,\\
W^{(a)}(\tilde{S},v,0)=0,\quad \tilde{S}_{max}\leq\tilde{S}<+\infty,\quad 0\leq v<+\infty.\\
\end{aligned}
\label{eq:PDEforABC1}
\end{equation}
We assume
\begin{equation}
W^{(a)}(\tilde{S}_{max},v,\tau)=\psi(v,\tau),\quad 0\leq v<+\infty, \quad 0\leq\tau\leq T,
\label{eq:PDEforABCpart2}
\end{equation}
with $\psi(v,0)=0$ to make the problem \eqref{eq:PDEforABC1} complete.
Let
\begin{equation}
\begin{aligned}
x=\ln{\tilde{S}},\quad
x_{max}=\ln{\tilde{S}_{max}},\quad
\gamma^{(a)}=\frac{1}{2}v\tau,\\
\tilde{W}^{(a)}(x,v,\gamma^{(a)})=e^{-\frac{1}{2}x+\frac{1}{4}\gamma^{(a)}}W^{(a)}(\tilde{S},v,\tau),
\end{aligned}
\label{eq:ABCtrans}
\end{equation}
then the problem \eqref{eq:PDEforABC1}-\eqref{eq:PDEforABCpart2} is equivalent to
\begin{equation}
\begin{aligned}
\frac{\partial^2\tilde{W}^{(a)}}{\partial x^2}=\frac{\partial \tilde{W}^{(a)}}{\partial\gamma^{(a)}},\quad
x_{max}<x<+\infty,\quad 0<v<+\infty,\quad 0<\gamma^{(a)}\leq \frac{1}{2}vT,\\
\tilde{W}^{(a)}(x,v,0)=0,\quad x_{max}\leq x<+\infty,\quad 0\leq v<+\infty,\\
\tilde{W}^{(a)}(x_{max},v,\gamma^{(a)})=e^{-\frac{1}{2}x_{max}+\frac{1}{4}\gamma^{(a)}}\psi(v,\frac{2}{v}\gamma^{(a)}),\quad
0\leq v<+\infty, \quad 0\leq\gamma^{(a)}\leq \frac{1}{2}vT.\\
\end{aligned}
\end{equation}
The solution can be given by
\begin{equation}
\tilde{W}^{(a)}(x,v,\gamma^{(a)})=\int_{0}^{\gamma^{(a)}}\frac{x-x_{max}}{\sqrt{4\pi(\gamma^{(a)}-s)^3}}e^{-\frac{(x-x_{max})^2}{4(\gamma^{(a)}-s)}}e^{-\frac{1}{2}x_{max}+\frac{1}{4}s}\psi(v,\frac{2}{v}s)ds.
\end{equation}
Let
\[\mu=\frac{x-x_{max}}{2\sqrt{\gamma^{(a)}-s}},\]
then
\begin{equation}
\tilde{W}^{(a)}(x,v,\gamma^{(a)})=\frac{2}{\sqrt{\pi}}\int_{\frac{x-x_{max}}{2\sqrt{\gamma^{(a)}}}}^{+\infty}e^{-\mu^2-\frac{1}{4}\left(\frac{x-x_{max}}{2\mu}\right)^2+\frac{1}{4}\gamma^{(a)}-\frac{1}{2}x_{max}}
\psi\left(v,\frac{2}{v}\left[\gamma^{(a)}-\left(\frac{x-x_{max}}{2\mu}\right)^2\right]\right)d\mu,
\end{equation}
and
\begin{equation}
\begin{aligned}
&\frac{\partial\tilde{W}^{(a)}}{\partial x}(x,v,\gamma^{(a)})\\
=&-\frac{2}{\sqrt{\pi}}\int_{\frac{x-x_{max}}{2\sqrt{\gamma^{(a)}}}}^{+\infty}e^{-\mu^2-\frac{1}{4}\left(\frac{x-x_{max}}{2\mu}\right)^2+\frac{1}{4}\gamma^{(a)}-\frac{1}{2}x_{max}}\left(\frac{x-x_{max}}{8\mu^2}\right)\psi\left(v,\frac{2}{v}\left[\gamma^{(a)}-\left(\frac{x-x_{max}}{2\mu}\right)^2\right]\right)d\mu\\
&+\frac{2}{\sqrt{\pi}}\int_{\frac{x-x_{max}}{2\sqrt{\gamma^{(a)}}}}^{+\infty}e^{-\mu^2-\frac{1}{4}\left(\frac{x-x_{max}}{2\mu}\right)^2+\frac{1}{4}\gamma^{(a)}-\frac{1}{2}x_{max}}\frac{\partial \psi}{\partial \tau}\left(v,\frac{2}{v}\left[\gamma^{(a)}-\left(\frac{x-x_{max}}{2\mu}\right)^2\right]\right)\left(-\frac{1}{v}\frac{x-x_{max}}{\mu^2}\right)d\mu,\\
\end{aligned}
\end{equation}
due to $\psi(v,0)=0$.
Let
\[s=\gamma^{(a)}-\left(\frac{x-x_{max}}{2\mu}\right)^2,\]
then
\begin{equation}
\frac{\partial\tilde{W}^{(a)}}{\partial x}(x,v,\gamma^{(a)})=-\frac{1}{\sqrt{\pi}}\int_{0}^{\gamma^{(a)}}e^{-\frac{(x-x_{max})2}{4(\gamma^{(a)}-s)}}e^{\frac{1}{4}s-\frac{1}{2}x_{max}}\left[\frac{1}{4}\psi\left(v,\frac{2}{v}s\right)\right.\left.+\frac{2}{v}\frac{\partial \psi}{\partial \tau}\left(v,\frac{2}{v}s\right)\right]\frac{1}{\sqrt{\gamma^{(a)}-s}}ds.\\
\end{equation}
Thus,
\begin{equation}
\frac{\partial\tilde{W}^{(a)}}{\partial x}(x_{max},v,\gamma^{(a)})=-\frac{1}{\sqrt{\pi}}\int_{0}^{\gamma^{(a)}}e^{\frac{1}{4}s-\frac{1}{2}x_{max}}\left[\frac{1}{4}\psi\left(v,\frac{2}{v}s\right)+\frac{2}{v}\frac{\partial \psi}{\partial \tau}\left(v,\frac{2}{v}s\right)\right]\frac{1}{\sqrt{\gamma^{(a)}-s}}ds.
\end{equation}
Do the inverse transformation of \eqref{eq:ABCtrans} and let
\[s=\frac{1}{2}v\tau^\prime,\]
then we have
\begin{equation}
\begin{aligned}
&\frac{\partial W^{(a)}}{\partial \tilde{S}}(\tilde{S}_{max},v,\tau)\\
=&-\frac{1}{\tilde{S}_{max}}\sqrt{\frac{v}{2\pi}}\int_{0}^{\tau}e^{-\frac{1}{8}v(\tau-\tau^\prime)}\left[\frac{1}{4}\psi(v,\tau^\prime)+\frac{2}{v}\frac{\partial \psi}{\partial \tau}(v,\tau^\prime)\right]\frac{1}{\sqrt{\tau-\tau^\prime}}d\tau^\prime+\frac{1}{2\tilde{S}_{max}}W^{(a)}(\tilde{S}_{max},v,\tau)\\
=&-\frac{1}{\tilde{S}_{max}}\sqrt{\frac{v}{2\pi}}\int_{0}^{\tau}e^{-\frac{1}{8}v(\tau-\tau^\prime)}\left[\frac{1}{4}W^{(a)}(\tilde{S}_{max},v,\tau^\prime)+\frac{2}{v}\frac{\partial W^{(a)}}{\partial \tau}(\tilde{S}_{max},v,\tau^\prime)\right]\frac{1}{\sqrt{\tau-\tau^\prime}}d\tau^\prime+\frac{1}{2\tilde{S}_{max}}W^{(a)}(\tilde{S}_{max},v,\tau).\\
\end{aligned}
\label{eq:ABClast}
\end{equation}
Since
\[V^{(a)}(\tilde{S},v,\tau)=W^{(a)}(\tilde{S},v,\tau)+\tilde{S}-1,\]
\begin{equation}
\begin{aligned}
\frac{\partial V^{(a)}}{\partial \tilde{S}}(\tilde{S}_{max},v,\tau)=&-\frac{1}{\tilde{S}_{max}}\sqrt{\frac{v}{2\pi}}\int_{0}^{\tau}e^{-\frac{1}{8}v(\tau-\tau^\prime)}\left[\frac{1}{4}V^{(a)}(\tilde{S}_{max},v,\tau^\prime)+\frac{2}{v}\frac{\partial V^{(a)}}{\partial \tau}(\tilde{S}_{max},v,\tau^\prime)\right]\frac{1}{\sqrt{\tau-\tau^\prime}}d\tau^\prime\\
&+\frac{1}{2\tilde{S}_{max}}V^{(a)}(\tilde{S}_{max},v,\tau)+\frac{\tilde{S}_{max}+1}{2\tilde{S}_{max}}+\frac{\tilde{S}_{max}-1}{4\tilde{S}_{max}}\sqrt{\frac{v}{2\pi}}\int_{0}^{\tau}e^{-\frac{1}{8}v(\tau-\tau^\prime)}\frac{1}{\sqrt{\tau-\tau^\prime}}d\tau^\prime.\\
\end{aligned}
\end{equation}
After some calculation, we have
\begin{equation}
\int^{\beta}_{\alpha}e^{-\frac{1}{8}v(\tau-\tau^\prime)}\frac{1}{\sqrt{\tau-\tau^\prime}}d\tau^\prime
=\frac{4\sqrt{2\pi}}{\sqrt{v}}\left[N\left(\frac{\sqrt{v(\tau-\alpha)}}{2}\right)-N\left(\frac{\sqrt{v(\tau-\beta)}}{2}\right)\right], \  \forall \alpha,\  \beta\in\mathbb{R},
\label{eq:calculation1}
\end{equation}
and
\begin{equation}
\int_{0}^{\tau}e^{-\frac{1}{8}v(\tau-\tau^\prime)}\frac{1}{\sqrt{\tau-\tau^\prime}}d\tau^\prime=\frac{4\sqrt{2\pi}}{\sqrt{v}}\left[N\left(\frac{\sqrt{v\tau}}{2}\right)-\frac{1}{2}\right],
\label{eq:calculation2}
\end{equation}
due to $N(0)=\frac{1}{2}$.
Then
\begin{equation}
\begin{aligned}
\frac{\partial V^{(a)}}{\partial \tilde{S}}(\tilde{S}_{max},v,\tau)=&-\frac{1}{\tilde{S}_{max}}\sqrt{\frac{v}{2\pi}}\int_{0}^{\tau}e^{-\frac{1}{8}v(\tau-\tau^\prime)}\left[\frac{1}{4}V^{(a)}(\tilde{S}_{max},v,\tau^\prime)+\frac{2}{v}\frac{\partial V^{(a)}}{\partial \tau}(\tilde{S}_{max},v,\tau^\prime)\right]\frac{1}{\sqrt{\tau-\tau^\prime}}d\tau^\prime\\
&+\frac{1}{2\tilde{S}_{max}}V^{(a)}(\tilde{S}_{max},v,\tau)+\frac{1}{\tilde{S}_{max}}+\frac{\tilde{S}_{max}-1}{\tilde{S}_{max}}N\left(\frac{\sqrt{v\tau}}{2}\right).
\end{aligned}
\label{eq:ABC}
\end{equation}
We call \eqref{eq:ABC} the approximate artificial boundary condition (ApABC) for the problem \eqref{eq:HestonProfirst}-\eqref{eq:HestonProlast}.

%--------------------
\subsection{The first modified approximate artificial boundary condition (MApABC1)}
We assume that the three derivatives w.r.t. $v$ in the PDE \eqref{eq:HestonProfirst} approximately make up a source term  independent of  $\tilde{S}$, which is denoted by $Q_1(v,\tau)$. Then the problem should be
\begin{equation}
\begin{aligned}
\frac{1}{2}v\tilde{S}^2\frac{\partial^2V^{(m1)}}{\partial \tilde{S}^2}+Q_1(v,\tau)=\frac{\partial V^{(m1)}}{\partial\tau},\quad
\tilde{S}_{max}<\tilde{S}<+\infty,\quad 0<v<+\infty,\quad 0<\tau\leq T,\\
V^{(m1)}(\tilde{S},v,0)=\tilde{S}-1,\quad \tilde{S}_{max}\leq\tilde{S}<+\infty,\quad 0\leq v<+\infty,
\label{eq:PDEforABCc1}
\end{aligned}
\end{equation}
where
\begin{equation}
Q_1(v,\tau)\approx \rho\sigma v\tilde{S}\frac{\partial^2V^{(m1)}}{\partial \tilde{S}\partial v}+\frac{1}{2}\sigma^2v\frac{\partial^2V^{(m1)}}{\partial v^2}+\kappa\left(\eta-v\right)\frac{\partial V^{(m1)}}{\partial v},
\label{eq:ABCc1Q}
\end{equation}
by the assumption.
Let
\[\tilde{V}^{(m1)}(\tilde{S},v,\tau)=V^{(m1)}(\tilde{S},v,\tau)-\int_{0}^{\tau}Q_1(v,s)ds,\]
then we can easily know that the ApABC in \eqref{eq:ABC} can be used on $\tilde{V}^{(m1)}$. And for $V^{(m1)}$, the boundary condition can be expressed as follows:
\begin{equation}
\begin{aligned}
\frac{\partial V^{(m1)}}{\partial \tilde{S}}(\tilde{S}_{max},v,\tau)=&-\frac{1}{\tilde{S}_{max}}\sqrt{\frac{v}{2\pi}}\int_{0}^{\tau}e^{-\frac{1}{8}v(\tau-\tau^\prime)}\left[\frac{1}{4}V^{(m1)}(\tilde{S}_{max},v,\tau^\prime)+\frac{2}{v}\frac{\partial V^{(m1)}}{\partial \tau}(\tilde{S}_{max},v,\tau^\prime)\right]\frac{1}{\sqrt{\tau-\tau^\prime}}d\tau^\prime\\
&+\frac{1}{2\tilde{S}_{max}}V^{(m1)}(\tilde{S}_{max},v,\tau)+\frac{1}{\tilde{S}_{max}}+\frac{\tilde{S}_{max}-1}{\tilde{S}_{max}}N\left(\frac{\sqrt{v\tau}}{2}\right)+ I_1(v,\tau),
\end{aligned}
\label{eq:ABCc1pro1}
\end{equation}
where
\begin{equation}
I_1(v,\tau)=\frac{1}{\tilde{S}_{max}}\sqrt{\frac{v}{2\pi}}\int_{0}^{\tau}e^{-\frac{1}{8}v(\tau-\tau^\prime)}\left[\frac{1}{4}\int_{0}^{\tau^\prime}Q_1(v,s)ds+\frac{2}{v}Q_1(v,\tau^\prime)\right]\frac{1}{\sqrt{\tau-\tau^\prime}}d\tau^\prime-\frac{1}{2\tilde{S}_{max}}\int_{0}^{\tau}Q_1(v,\tau^\prime)d\tau^\prime.\nonumber
\end{equation}
After some calculation, we get
\begin{equation}
I_1(v,\tau)=\frac{1}{\tilde{S}_{max}}\int_{0}^{\tau}\left[\frac{\sqrt{2}}{\sqrt{\pi v(\tau-\tau^\prime)}}e^{-\frac{1}{8}v(\tau-\tau^\prime)}+N\left(\frac{\sqrt{v(\tau-\tau^\prime)}}{2}\right)-1\right]Q_1(v,\tau^\prime)d\tau^\prime.
\label{eq:ABCc1l1}
\end{equation}
We call \eqref{eq:ABCc1pro1}-\eqref{eq:ABCc1l1} the first modified approximate artificial boundary condition (MApABC1).

%----------------------
\subsection{The second modified approximate artificial boundary condition (MApABC2)}
We denote the three derivatives w.r.t $v$ in the PDE  \eqref{eq:HestonProfirst} by $Q_2(\tilde{S}, v,\tau)$. The problem should be
\begin{equation}
\begin{aligned}
\frac{1}{2}v\tilde{S}^2\frac{\partial^2V^{(m2)}}{\partial \tilde{S}^2}+Q_2(\tilde{S}, v,\tau)=\frac{\partial V^{(m2)}}{\partial\tau},\quad
\tilde{S}_{max}<\tilde{S}<+\infty,\quad 0<v<+\infty,\quad 0<\tau\leq T,\\
V^{(m2)}(\tilde{S},v,0)=\tilde{S}-1,\quad \tilde{S}_{max}\leq\tilde{S}<+\infty,\quad 0\leq v<+\infty,
\label{eq:PDEforABCc2}
\end{aligned}
\end{equation}
where
\begin{equation}
Q_2(\tilde{S},v,\tau)\approx \rho\sigma v\tilde{S}\frac{\partial^2V^{(m2)}}{\partial \tilde{S}\partial v}+\frac{1}{2}\sigma^2v\frac{\partial^2V^{(m2)}}{\partial v^2}+\kappa\left(\eta-v\right)\frac{\partial V^{(m2)}}{\partial v}.
\label{eq:ABCc2Q2}
\end{equation}
Let
\begin{equation}
\begin{aligned}
\tilde{V}^{(m2)}(\tilde{S},v,\tau)=&V^{(m2)}(\tilde{S},v,\tau)-\int_{0}^{\tau}\!\!\int_{\tilde{S}_{max}}^{+\infty}\frac{1}{\sqrt{2\pi v(\tau-s)}}\left(\exp{\left(-\frac{(\ln\tilde{S}-\ln S^\prime)^2}{2v(\tau-s)}\right)}\right.\\
&\left.-\exp{\left(-\frac{(\ln\tilde{S}-2\ln\tilde{S}_{max}+\ln S^\prime)^2}{2v(\tau-s)}\right)}\right)\exp{\left(\frac{1}{2}(\ln\tilde{S} - 3\ln S^\prime)\right.}
{\left.-\frac{1}{8}v(\tau - s)\right)}Q_2(S^\prime,v,s)dS^\prime ds,\nonumber
\end{aligned}
\end{equation}
Similarly, the ApABC in \eqref{eq:ABC} can also be used on $\tilde{V}^{(m2)}$ and the boundary condition for $V^{(m2)}$ is shown below:
\begin{equation}
\begin{aligned}
\frac{\partial V^{(m2)}}{\partial \tilde{S}}(\tilde{S}_{max},v,\tau)=&-\frac{1}{\tilde{S}_{max}}\sqrt{\frac{v}{2\pi}}\int_{0}^{\tau}e^{-\frac{1}{8}v(\tau-\tau^\prime)}\left[\frac{1}{4}V^{(m2)}(\tilde{S}_{max},v,\tau^\prime)+\frac{2}{v}\frac{\partial V^{(m2)}}{\partial \tau}(\tilde{S}_{max},v,\tau^\prime)\right]\frac{1}{\sqrt{\tau-\tau^\prime}}d\tau^\prime\\
&+\frac{1}{2\tilde{S}_{max}}V^{(m2)}(\tilde{S}_{max},v,\tau)+\frac{1}{\tilde{S}_{max}}+\frac{\tilde{S}_{max}-1}{\tilde{S}_{max}}N\left(\frac{\sqrt{v\tau}}{2}\right)+ I_2(v,\tau),
\end{aligned}
\label{eq:ABCc2pro1}
\end{equation}
where, after some calculation,
\begin{equation}
\begin{aligned}
I_2(v,\tau)=&\frac{1}{\tilde{S}_{max}}\int_{0}^{\tau}\int_{\tilde{S}_{max}}^{+\infty}\frac{\sqrt{2}}{\sqrt{\pi v(\tau-\tau^\prime)}}\frac{\ln S^\prime - \ln\tilde{S}_{max}}{v(\tau-\tau^\prime)}\\
&\exp{\left(-\frac{(\ln S^\prime-\ln\tilde{S}_{max})^2}{2v(\tau-\tau^\prime)}-\frac{1}{2}(3\ln S^\prime-\ln\tilde{S}_{max})-\frac{1}{8}v(\tau-\tau^\prime)\right)}Q_2(S^\prime,v,\tau^\prime)dS^\prime d\tau^\prime.
\end{aligned}
\label{eq:ABCc2l2}
\end{equation}
We call \eqref{eq:ABCc2pro1}-\eqref{eq:ABCc2l2} the second modified approximate artificial boundary condition (MApABC2). Note that if we assume $Q_2$ is independent of $\tilde{S}$
, i.e. $Q_2(\tilde{S},v,\tau)=Q_1(v,\tau)$,
then MApABC2 in \eqref{eq:ABCc2pro1}-\eqref{eq:ABCc2l2} can be easily reduced to MApABC1 in \eqref{eq:ABCc1pro1}-\eqref{eq:ABCc1l1}.

%--%--%--%--%--%--%--%--%--%--%--%--%--%--%--%--%--%

\section{Finite difference schemes}\label{sec:fdm}
\label{section:fds}
In this section, we restrict the problem \eqref{eq:HestonProfirst}-\eqref{eq:HestonProlast} on a bounded domain $[0,T]\times[0,\tilde{S}_{max}]\times[0,v_{max}]$ with different artificial boundary conditions and build finite difference schemes to solve the corresponding problems numerically. We apply the two modified approximate artificial boundary conditions on $\tilde{S}=\tilde{S}_{max}$ proposed above respectively and a homogeneous Neumann boundary condition on $v=v_{max}$ as in \cite{Diamond} instead of the one by Heston which leads to singular jump. Before moving on to the problems with MApABC1 and MApABC2, we take the problem with ApABC as an example for simplicity:
\begin{eqnarray}
\frac{1}{2}v\tilde{S}^2\frac{\partial^2V}{\partial \tilde{S}^2}+\rho\sigma v\tilde{S}\frac{\partial^2V}{\partial \tilde{S}\partial v}+\frac{1}{2}\sigma^2v\frac{\partial^2V}{\partial v^2}+\kappa\left(\eta-v\right)\frac{\partial V}{\partial v}=\frac{\partial V}{\partial \tau},\nonumber\\
0<\tilde{S}<\tilde{S}_{max},\quad 0<v<v_{max},\quad 0<\tau\leq T,
\label{eq:HestonABCProfirst}\\
V(\tilde{S},v,0)=(\tilde{S}-1)^+,\quad 0\leq\tilde{S}\leq\tilde{S}_{max},\quad 0\leq v\leq v_{max},
\label{eq:HestonABCProsecond}\\
V\left(0,v,\tau\right)=0,\quad 0\leq v\leq v_{max},\quad 0<\tau\leq T,\\
\frac{\partial V}{\partial \tilde{S}}(\tilde{S}_{max},v,\tau)=-\frac{1}{\tilde{S}_{max}}\sqrt{\frac{v}{2\pi}}\int_{0}^{\tau}e^{-\frac{1}{8}v(\tau-\tau^\prime)}\left[\frac{1}{4}V(\tilde{S}_{max},v,\tau^\prime)
+\frac{2}{v}\frac{\partial V}{\partial \tau}(\tilde{S}_{max},v,\tau^\prime)\right]\frac{1}{\sqrt{\tau-\tau^\prime}}d\tau^\prime\nonumber\\
+\frac{1}{2\tilde{S}_{max}}V(\tilde{S}_{max},v,\tau)+\frac{1}{\tilde{S}_{max}}
+\frac{\tilde{S}_{max}-1}{\tilde{S}_{max}}N\left(\frac{\sqrt{v\tau}}{2}\right),\quad 0\leq v\leq v_{max}, \quad 0<\tau\leq T,
\label{eq:HestonABCforS}\\
\kappa\eta\frac{\partial V}{\partial v}(\tilde{S},0,\tau)=\frac{\partial V}{\partial \tau}(\tilde{S},0,\tau),\quad 0\leq \tilde{S}\leq\tilde{S}_{max},\quad 0<\tau\leq T,
\label{eq:HestonABCProlast2}\\
\frac{\partial V}{\partial v}(\tilde{S},v_{max},\tau)=0,\quad 0\leq \tilde{S}\leq\tilde{S}_{max},\quad 0<\tau\leq T.
\label{eq:HestonABCProlast}
\end{eqnarray}

We construct a uniform grid with
\[\tau_n=n\Delta\tau,\quad n=0, ..., N,\quad \Delta\tau=\frac{T}{N},\]
\[\tilde{S}_i=i\Delta \tilde{S},\quad i=0,...,I,\quad \Delta \tilde{S}=\frac{\tilde{S}_{max}}{I},\]
\[v_j=j\Delta v,\quad j=0,...,J,\quad \Delta v=\frac{v_{max}}{J},\]
and denote
\[V_{i,j}^n=V(\tilde{S}_i,v_j,\tau_n),\quad i=0,...,I, \quad j=0,...,J,\quad n=0,...,N.\]

Due to the nonsmoothness of the initial condition in \eqref{eq:HestonABCProsecond}, we choose the backward Euler scheme for the first time step and then turn to the Crank-Nicolson scheme for the subsequent time steps for both the PDE \eqref{eq:HestonABCProfirst} and the boundary condition on $v=0$ in \eqref{eq:HestonABCProlast2}. We only present the Crank-Nicolson scheme below for simplicity.

For the PDE \eqref{eq:HestonABCProfirst}, we use a central difference scheme for the second-order derivative term w.r.t. $\tilde{S}$ and the cross-derivative term and a Samarskii scheme for the first- and second-order derivative terms w.r.t. $v$ which is upwind and second-order accurate,
\begin{equation}
\begin{split}
\frac{1}{4}v_j\tilde{S}_i^2\frac{V_{i+1,j}^{n+1}-2V_{i,j}^{n+1}+V_{i-1,j}^{n+1}}{(\Delta\tilde{S})^2}+\frac{1}{2}\rho\sigma v_j\tilde{S}_i\frac{V_{i+1,j+1}^{n+1}-V_{i-1,j+1}^{n+1}+V_{i-1,j-1}^{n+1}-V_{i+1,j-1}^{n+1}}{4\Delta\tilde{S}\Delta v}\\
+\frac{1}{4}\sigma^2v_j\frac{1}{1+R_j}\frac{V_{i,j+1}^{n+1}-2V_{i,j}^{n+1}+V_{i,j-1}^{n+1}}{(\Delta v)^2}
+\frac{1}{2}\kappa\left(\eta-v_j\right)^+\frac{V_{i,j+1}^{n+1}-V_{i,j}^{n+1}}{\Delta v}+\frac{1}{2}\kappa\left(\eta-v_j\right)^-\frac{V_{i,j}^{n+1}-V_{i,j-1}^{n+1}}{\Delta v}\\
+\frac{1}{4}v_j\tilde{S}_i^2\frac{V_{i+1,j}^{n}-2V_{i,j}^{n}+V_{i-1,j}^{n}}{(\Delta\tilde{S})^2}+\frac{1}{2}\rho\sigma v_j\tilde{S}_i\frac{V_{i+1,j+1}^{n}-V_{i-1,j+1}^{n}+V_{i-1,j-1}^{n}-V_{i+1,j-1}^{n}}{4\Delta\tilde{S}\Delta v}\\
+\frac{1}{4}\sigma^2v_j\frac{1}{1+R_j}\frac{V_{i,j+1}^{n}-2V_{i,j}^{n}+V_{i,j-1}^{n}}{(\Delta v)^2}
+\frac{1}{2}\kappa\left(\eta-v_j\right)^+\frac{V_{i,j+1}^{n}-V_{i,j}^n}{\Delta v}+\frac{1}{2}\kappa\left(\eta-v_j\right)^-\frac{V_{i,j}^{n}-V_{i,j-1}^n}{\Delta v}\\
=\frac{V_{i,j}^{n+1}-V_{i,j}^{n}}{\Delta\tau},\\
\end{split}
\end{equation}
where
\[R_j=\frac{\kappa|\eta-v_j|}{\sigma^2v_j}\Delta v,\quad i=1,...,I-1,\quad j=1,...,J-1, \quad n=1,...,N-1.\]

For the boundary condition \eqref{eq:HestonABCProlast2} on $v=0$, we select the upwind scheme similarly:
\begin{equation}
\frac{1}{2}\kappa\eta\frac{V_{i,1}^{n+1}-V_{i,0}^{n+1}}{\Delta v}+\frac{1}{2}\kappa\eta\frac{V_{i,1}^{n}-V_{i,0}^n}{\Delta v}=\frac{V_{i,0}^{n+1}-V_{i,0}^{n}}{\Delta\tau},\quad i=1,...,I,\quad n=1,...,N.
\end{equation}

A constant extrapolation is used for the last boundary condition \eqref{eq:HestonABCProlast}:
\begin{equation}
V_{i,J}^{n}=V_{i,J-1}^{n},\quad i=1,...,I, \quad n=1,...,N.
\end{equation}

The initial condition \eqref{eq:HestonABCProsecond} is implemented by integral average.

Next, we come to the ApABC on $\Gamma_{\tilde{S}}$ in \eqref{eq:HestonABCforS}. The main problem here is to approximate the integral in the formula. For $\tau=\tau_n(n\in\{2,...,N\})$, $v=v_j(j\in\{1,...,J-1\})$, we first partition the interval $[0,\tau_n]$ into $[0, \tau_{n-1}]$ and $[\tau_{n-1}, \tau_{n}]$ and make a change of variable to deal with the singularity in the second part. Then we can directly apply the trapezoidal rule to calculate the integrals. The main steps and results are shown below:
\begin{equation}
\begin{aligned}
&\int_{0}^{\tau_{n-1}}e^{-\frac{1}{8}v_j(\tau_n-\tau^\prime)}\left[\frac{1}{4}V(\tilde{S}_{max},v_j,\tau^\prime)+\frac{2}{v}\frac{\partial V}{\partial \tau}(\tilde{S}_{max},v_j,\tau^\prime)\right]\frac{1}{\sqrt{\tau_n-\tau^\prime}}d\tau^\prime\\
\approx&\Delta\tau\left\{\frac{1}{2} e^{-\frac{1}{8}v_j\tau_{n}}\frac{1}{4}(\tilde{S}_{max}-1)\frac{1}{\sqrt{\tau_n}}+\sum_{k=1}^{n-2} e^{-\frac{1}{8}v_j\tau_{n-k}}\left[\frac{1}{4}V_{I,j}^{k}+\frac{2}{v_j\Delta\tau}(V^{k}_{I,j}-V^{k-1}_{I,j})\right]\frac{1}{\sqrt{\tau_{n-k}}}\right.\\
&\left.+\frac{1}{2}e^{-\frac{1}{8}v_j\tau_{1}}\left[\frac{1}{4}V_{I,j}^{n-1}+\frac{2}{v_j\Delta\tau}(V^{n-1}_{I,j}-V^{n-2}_{I,j})\right]\frac{1}{\sqrt{\tau_{1}}}\right\},
\end{aligned}
\end{equation}
by
\[V(\tilde{S}_{max},v,0)=\tilde{S}_{max}-1,\quad \frac{\partial V}{\partial \tau}(\tilde{S}_{max},v,0)=0,\quad\forall 0\leq v<+\infty,\]
with Theorem \ref{thmAsym3}.
Let
\[s=\sqrt{\tau_n-\tau^\prime},\]
then
\begin{equation}
\begin{split}
&\int_{\tau_{n-1}}^{\tau_{n}}e^{-\frac{1}{8}v_j(\tau_n-\tau^\prime)}\left[\frac{1}{4}V(\tilde{S}_{max},v_j,\tau^\prime)+\frac{2}{v}\frac{\partial V}{\partial \tau}(\tilde{S}_{max},v_j,\tau^\prime)\right]\frac{1}{\sqrt{\tau_n-\tau^\prime}}d\tau^\prime\\
=&2\int_{0}^{\sqrt{\Delta\tau}}e^{-\frac{1}{8}v_j s^2}\left[\frac{1}{4}V(\tilde{S}_{max},v_j,\tau_n-s^2)+\frac{2}{v}\frac{\partial V}{\partial \tau}(\tilde{S}_{max},v_j,\tau_n-s^2)\right]ds\\
\approx&\sqrt{\Delta\tau}\left\{\left[\frac{1}{4}V_{I,j}^{n}+\frac{2}{v\Delta\tau}(V_{I,j}^{n}-V_{I,j}^{n-1})\right]+e^{-\frac{1}{8}v_j \Delta\tau}\left[\frac{1}{4}V_{I,j}^{n-1}+\frac{2}{v\Delta\tau}(V_{I,j}^{n-1}-V_{I,j}^{n-2})\right]\right\}.
\end{split}
\end{equation}
We can easily get to the following approximation of ApABC:
\begin{equation}
\begin{aligned}
&\frac{V_{I,j}^{n}-V_{I-1,j}^{n}}{\Delta\tilde{S}}\\
=&-\frac{\tilde{S}_{max}-1}{\tilde{S}_{max}}\frac{1}{8}\sqrt{\frac{v_j\Delta\tau}{2n\pi}}e^{-\frac{1}{8}v_j\tau_{n}}\\
&-\frac{1}{\tilde{S}_{max}}\sqrt{\frac{v_j}{2\pi}}\left\{\sum_{k=1}^{n-2}e^{-\frac{1}{8}v_j\tau_{n-k}}\left[\left(\frac{1}{4}\sqrt{\Delta\tau}+\frac{2}{v_j\sqrt{\Delta\tau}}\right)V_{I,j}^{k}-\frac{2}{v_j\sqrt{\Delta\tau}}V^{k-1}_{I,j}\right]\frac{1}{\sqrt{n-k}}\right.\\
&\left.+\frac{1}{2}e^{-\frac{1}{8}v_j \Delta\tau}\left[\left(\frac{1}{4}\sqrt{\Delta\tau}+\frac{2}{v_j\sqrt{\Delta\tau}}\right)V_{I,j}^{n-1}-\frac{2}{v_j\sqrt{\Delta\tau}}V_{I,j}^{n-2}\right]\right.\\
&\left.+e^{-\frac{1}{8}v_j \Delta\tau}\left[\left(\frac{1}{4}\sqrt{\Delta\tau}+\frac{2}{v_j\sqrt{\Delta\tau}}\right)V_{I,j}^{n-1}-\frac{2}{v_j\sqrt{\Delta\tau}}V_{I,j}^{n-2}\right]\right.\\
&\left.+\left[\left(\frac{1}{4}\sqrt{\Delta\tau}+\frac{2}{v_j\sqrt{\Delta\tau}}\right)V_{I,j}^{n}-\frac{2}{v_j\sqrt{\Delta\tau}}V_{I,j}^{n-1}\right]\right\}\\
&+\frac{1}{2\tilde{S}_{max}}V_{I,j}^{n}+\frac{1}{\tilde{S}_{max}}+\frac{\tilde{S}_{max}-1}{\tilde{S}_{max}}N\left(\frac{\sqrt{v_j\tau_n}}{2}\right),\quad j=1,..,J-1, \quad n=2,...,N.
\end{aligned}
\end{equation}
Furthermore, it can be rewritten as
\begin{equation}
\begin{aligned}
&\left[1+\frac{\Delta\tilde{S}}{\tilde{S}_{max}}\left[\frac{1}{\sqrt{2\pi}}\left(\frac{1}{4}\sqrt{v_j\Delta\tau}+\frac{2}{\sqrt{v_j\Delta\tau}}\right)-\frac{1}{2}\right]\right]V_{I,j}^{n}-V_{I-1,j}^{n}\\
=&-\frac{\Delta\tilde{S}}{\tilde{S}_{max}}\frac{1}{\sqrt{2\pi}}\left\{\sum_{k=1}^{n-3}\left[e^{-\frac{1}{8}v_j\tau_{n-k}}\left(\frac{1}{4}\sqrt{v_j\Delta\tau}+\frac{2}{\sqrt{v_j\Delta\tau}}\right)\frac{1}{\sqrt{n-k}}-e^{-\frac{1}{8}v_j\tau_{n-k-1}}\frac{2}{\sqrt{v_j\Delta\tau}}\frac{1}{\sqrt{n-k-1}}\right]V_{I,j}^{k}\right.\\
&+\left[e^{-\frac{1}{8}v_j\tau_{2}}\left(\frac{1}{4}\sqrt{v_j\Delta\tau}+\frac{2}{\sqrt{v_j\Delta\tau}}\right)\frac{1}{\sqrt{2}}-\frac{3}{2}e^{-\frac{1}{8}v_j \tau_1}\frac{2}{\sqrt{v_j\Delta\tau}}\right]V_{I,j}^{n-2}\\
&\left.+\left[\frac{3}{2}e^{-\frac{1}{8}v_j \tau_1}\left(\frac{1}{4}\sqrt{v_j\Delta\tau}+\frac{2}{\sqrt{v_j\Delta\tau}}\right)-\frac{2}{\sqrt{v_j\Delta\tau}}\right]V_{I,j}^{n-1}\right\}\\
&-\Delta\tilde{S}\frac{\tilde{S}_{max}-1}{\tilde{S}_{max}}\frac{1}{\sqrt{2\pi}}\left(\frac{1}{2}e^{-\frac{1}{8}v_j\tau_{n}}\frac{1}{4}\sqrt{v_j\Delta\tau}\frac{1}{
\sqrt{n}}-e^{-\frac{1}{8}v_j\tau_{n-1}}\frac{2}{\sqrt{v_j\Delta\tau}}\frac{1}{\sqrt{n-1}}\right)\\
&+\frac{\Delta\tilde{S}}{\tilde{S}_{max}}+\Delta\tilde{S}\frac{\tilde{S}_{max}-1}{\tilde{S}_{max}}N\left(\frac{\sqrt{v_j\tau_n}}{2}\right),\quad j=1,..,J-1, \quad n=3,...,N,
\end{aligned}
\end{equation}
and
\begin{equation}
\begin{aligned}
&\left[1+\frac{\Delta\tilde{S}}{\tilde{S}_{max}}\left[\frac{1}{\sqrt{2\pi}}\left(\frac{1}{4}\sqrt{v_j\Delta\tau}+\frac{2}{\sqrt{v_j\Delta\tau}}\right)-\frac{1}{2}\right]\right]V_{I,j}^{2}-V_{I-1,j}^{2}\\
=&-\frac{\Delta\tilde{S}}{\tilde{S}_{max}}\frac{1}{\sqrt{2\pi}}
\left[\frac{3}{2}e^{-\frac{1}{8}v_j \tau_1}\left(\frac{1}{4}\sqrt{v_j\Delta\tau}+\frac{2}{\sqrt{v_j\Delta\tau}}\right)-\frac{2}{\sqrt{v_j\Delta\tau}}\right]V_{I,j}^{1}\\
&-\Delta\tilde{S}\frac{\tilde{S}_{max}-1}{\tilde{S}_{max}}\frac{1}{\sqrt{2\pi}}\left(\frac{1}{2}e^{-\frac{1}{8}v_j\tau_{2}}\frac{1}{4}\sqrt{v_j\Delta\tau}\frac{1}{\sqrt{2}}-\frac{3}{2}e^{-\frac{1}{8}v_j \tau_1}\frac{2}{\sqrt{v_j\Delta\tau}}\right)\\
&+\frac{\Delta\tilde{S}}{\tilde{S}_{max}}+\Delta\tilde{S}\frac{\tilde{S}_{max}-1}{\tilde{S}_{max}}N\left(\frac{\sqrt{v_j\tau_2}}{2}\right),\quad j=1,..,J-1.
\end{aligned}
\end{equation}
Similarly, when $n=1$, the boundary condition can be approximated by
\begin{equation}
\begin{aligned}
&\left[1+\frac{\Delta\tilde{S}}{\tilde{S}_{max}}\left[\frac{1}{\sqrt{2\pi}}\left(\frac{1}{4}\sqrt{v_j\Delta\tau}+\frac{2}{\sqrt{v_j\Delta\tau}}\right)-\frac{1}{2}\right]\right]V_{I,j}^{1}-V_{I-1,j}^{1}\\
=&-\Delta\tilde{S}\frac{\tilde{S}_{max}-1}{\tilde{S}_{max}}\frac{1}{\sqrt{2\pi}}\left(e^{-\frac{1}{8}v_j \tau_1}\frac{1}{4}\sqrt{v_j\Delta\tau}-\frac{2}{\sqrt{v_j\Delta\tau}}\right)+\frac{\Delta\tilde{S}}{\tilde{S}_{max}}+\Delta\tilde{S}\frac{\tilde{S}_{max}-1}{\tilde{S}_{max}}N\left(\frac{\sqrt{v_j\tau_1}}{2}\right),\quad j=1,..,J-1.
\end{aligned}
\end{equation}

To give the approximation of MApABC1  and MApABC2, we only need to deal with the two integrals $I_1(v,\tau)$ and $I_2(v,\tau)$ in \eqref{eq:ABCc1l1} and \eqref{eq:ABCc2l2} since the rest part of the two conditions are both the same as ApABC. We partition the domain of integration into two parts and apply the trapezoidal rule again. We do not go into details for this part but only show the main idea and the approximation of $Q_1(v,\tau)$ and $Q_2(\tilde{S},v,\tau)$ below.

For $Q_1(v_j,\tau_n):=Q_{1,j}^{n}(j\in\{1,...,J-1\},n\in\{1, ..., N\})$, based on \eqref{eq:ABCc1Q}, we use the  upwind scheme
\begin{equation}
\begin{aligned}
Q_{1,j}^{n}\approx& \rho\sigma v_j\tilde{S}_{max}\frac{V_{I,j+1}^n-V_{I-1,j+1}^n-V_{I,j}^n+V_{I-1,j}^n}{\Delta\tilde{S}\Delta v}+\frac{1}{2}\sigma^2v_j\frac{V_{I,j+1}^n-2V_{I,j}^n+V_{I,j-1}^n}{(\Delta v)^2}\\
&+\kappa\left(\eta-v_j\right)^+\frac{V_{I,j+1}^n-V_{I,j}^n}{\Delta v}+\kappa\left(\eta-v_j\right)^-\frac{V_{I,j}^n-V_{I,j-1}^n}{\Delta v},
\quad j=1,...,J-1, \quad n=1,...,N,
\end{aligned}
\end{equation}
and
\[Q_{1,j}^{0}\approx 0, \quad j=1,...,J-1.\]

Compared with $I_1(v,\tau)$, $I_2(v,\tau)$ is not straightforward to deal with due to the integral over $S^\prime \in [\tilde{S}_{max},+\infty)$. For $Q_2(\tilde{S}_i, v_j,\tau_n):=Q_{2,i,j}^{n}(i\in\{1, ..., I-1\},j\in\{1,...,J-1\},n\in\{1, ..., N\})$, based on \eqref{eq:ABCc1Q},we use the central difference scheme
\begin{equation}
\begin{aligned}
Q_{2,i,j}^{n}\approx& \rho\sigma v_j\tilde{S}_{i}\frac{V_{i+1,j+1}^n-V_{i-1,j+1}^n-V_{i+1,j-1}^n+V_{i-1,j-1}^n}{4\Delta\tilde{S}\Delta v}+\frac{1}{2}\sigma^2v_j\frac{V_{i,j+1}^n-2V_{i,j}^n+V_{i,j-1}^n}{(\Delta v)^2}\\
&+\kappa\left(\eta-v_j\right)\frac{V_{i,j+1}^n-V_{i,j-1}^n}{2\Delta v},
\quad i=1, ..., I-1, \quad j=1,...,J-1, \quad n=1,...,N,
\end{aligned}
\label{eq:discreteQ2}
\end{equation}
and
\[Q_{2,i,j}^{0}\approx 0, \quad i=1, ..., I-1,\quad j=1,...,J-1.\]
For $j=1,...,J-1, \ n=1,...,N$, we try to find a curve to fit  $Q_{2,i,j}^{n}$'s corresponding to $\tilde{S}_i$'s with $i=1, ..., I-1$. Then we calculate the integral w.r.t. $S^\prime$ with this curve and approximate the integral w.r.t. $\tau^\prime$ applying the trapezoidal rule in $I_2(v,\tau)$. Referring to the asymptotic solutions in Theorem \ref{thmAsymp}, we choose the curve
\[(c_0+c_1 \ln\tilde{S})\exp{\left(-\frac{(\ln\tilde{S}-\mu_f)^2}{2\sigma_f^2}\right)}\]
with the parameters $c_0, c_1, \mu_f, \sigma_f$ to be determined.

%--%--%--%--%--%--%--%--%--%--%--%--%--%--%--%--%--%
\section{Numerical experiments}\label{sec:num}
We perform numerical experiments and report the results in this section. We apply MApABC1 and MApABC2 on $\tilde{S}_{max}$ and compare the results with those using the original boundary condition by Heston.

The first set of parameters is shown in Table \ref{tabl:Parameters-Ex7}.
%Parameters
	\begin{table}[h]
		\centering
		\caption{Parameters}
		\label{tabl:Parameters-Ex7}
		\begin{tabular}{cr}
			\hline
			Parameter & Value \\
			\hline
			$\kappa$  & 4			\\
			$\eta$       & 0.1		\\
			$\sigma$  & 0.1		\\
			$\rho$ 		 & -0.5		\\
			$T$ 			 & 2			\\
			\hline
		\end{tabular}
	\end{table}
We compute the partial derivative of the 0-order asymptotic solution $V_0(\tilde{S},v,\tau)$ in \eqref{eq:HestonAsym0} w.r.t. $v$ to investigate approximately how the partial derivative of the solution $V(\tilde{S},v,\tau)$ behaves as $\tilde{S}$ and $v$ changes.
%dV/dv->vmax=4
	\begin{figure}[htbp!]
		\centering
			\includegraphics[width=0.9\linewidth]{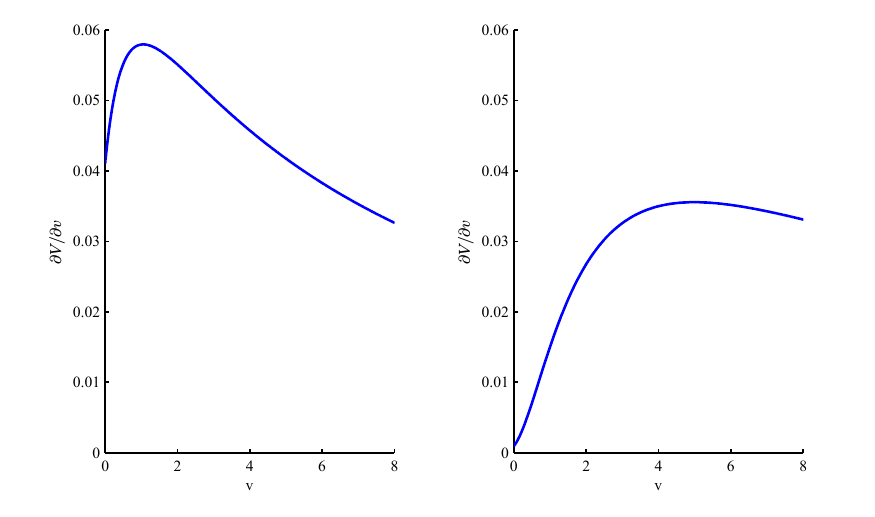}
		\caption{$\frac{\partial V_0}{\partial v}(\tilde{S},v,\tau)$ on $\tilde{S}=2$ (left) and $\tilde{S}=4$ (right) where $v$ ranges from 0 to 8.}
		\label{graph:VvSfixed-Ex7}
	\end{figure}

As is shown in Figure \ref{graph:VvSfixed-Ex7}, for a fixed $\tilde{S}$, $\frac{\partial V_0}{\partial v}(\tilde{S},v,\tau)$ increases first and then decreases with $v$ growing larger. When $\tilde{S}$ is small, $\frac{\partial V_0}{\partial v}(\tilde{S},v,\tau)$ increases faster and decreases faster as well than it does with a large $\tilde{S}$. We need to take into account both the accuracy and computational costs when looking for a proper set of $S_{max}$ and $v_{max}$. Since we apply the homogeneous Neumann boundary condition on $v_{max}$, we tend to choose a smaller $\tilde{S}_{max}$ and a larger $v_{max}$ to make the partial derivative w.r.t $v$ small enough. Besides, a large $\tilde{S}_{max}$ is needed to reduce the error caused by the boundary condition on the $\tilde{S}$-direction. In this case, we set
\[\tilde{S}_{max}=4, \quad v_{max}=4.\]
When $\tilde{S}=4$, we can see in Figure \ref{graph:VvSfixed-Ex7} that $v$ changes slightly when it is greater than 4.

Letting
\[\Delta\tau = \Delta \tilde{S}=\Delta v:=h,\]
we conduct experiments using the original boundary condition  by Heston, MApABC1 and MApABC2 and denote the solutions by $V^{(\Lambda)}$($\Lambda \in$ \{OriginalBC, MApABC1, MApABC2\}). We use the 2nd-order asymptotic solution in  \eqref{eq:HestonAsymall} denoted by $V^{(asym)}$
as a reference and calculate the 2-norm relative errors of the numerical solutions with respect to $V^{(asym)}$:
\[Error^{(\Lambda)}=\frac{\|V^{(\Lambda)}-V^{(asym)}\|_2}{\|V^{(asym)}\|_2},\]
where $\Lambda \in$ \{OriginalBC, MApABC1, MApABC2\}. We report the errors corresponding to different steps in Table \ref{tabl:Errors-Ex7}. We find that all of the three errors decrease as the step gets smaller but they all change more and more slightly. The total error may be dominated by the systematic error when the step $h$ is small enough. Comparing the three boundary conditions, we find that MApABC1 behaves better than the original boundary condition by Heston and MApABC2 is the best among the three in terms of the error. The error is reduced by about 61\% when we replace the original boundary condition by Heston with MApABC1 for the step $h=0.4$. When $h$ gets smaller, the reduction ratio decreases and gets to about 51\% when $h=0.025$. With MApABC2, the error is reduced by 68\% compared with the original boundary condition by Heston when $h =0.4$ and the ratio increases to 97\% when $h=0.025$.

%Error
	\begin{table}[h]
		\centering
		\caption{Relative errors}
		\label{tabl:Errors-Ex7}
		\begin{tabular}{crrrrrr}
			\hline
			Steps 			& $h=0.4$ & $h=0.2$ & $h=0.1$ & $h=0.05$ & $h=0.025$\\
			\hline
			OriginalBC & 0.01223 & 0.00929 & 0.00827 & 0.00787 & 0.00768\\
			MApABC1 & 0.00478 & 0.00395 & 0.00386 & 0.00382 & 0.00377\\
			MApABC2 & 0.00396& 0.00156 & 0.00063 & 0.00033 & 0.00020\\
			\hline
		\end{tabular}
	\end{table}

Fix $h=0.1$. We have mentioned in the last section that to approximate the integral $I_2(v,\tau)$ in MApABC2, we select a class of curves to fit the discrete values of $Q_2(\tilde{S},v,\tau)$ computed by \eqref{eq:discreteQ2}. Here, we draw the discrete and fitted values of  $Q_2(\tilde{S},v,\tau)$ on both $v=0.8, \tau=0.4$ and $v=3.6, \tau=1.5$ as examples in Figure \ref{graph:fit-Ex7}. It is easy to find that the class of curves we have chosen can give a good fit of the discrete value of $Q_2(\tilde{S},v,\tau)$ so that we can obtain a satisfying approximation of $I_2(v,\tau)$.

	\begin{figure}[htbp!]
		\centering
			\includegraphics[width=0.9\linewidth]{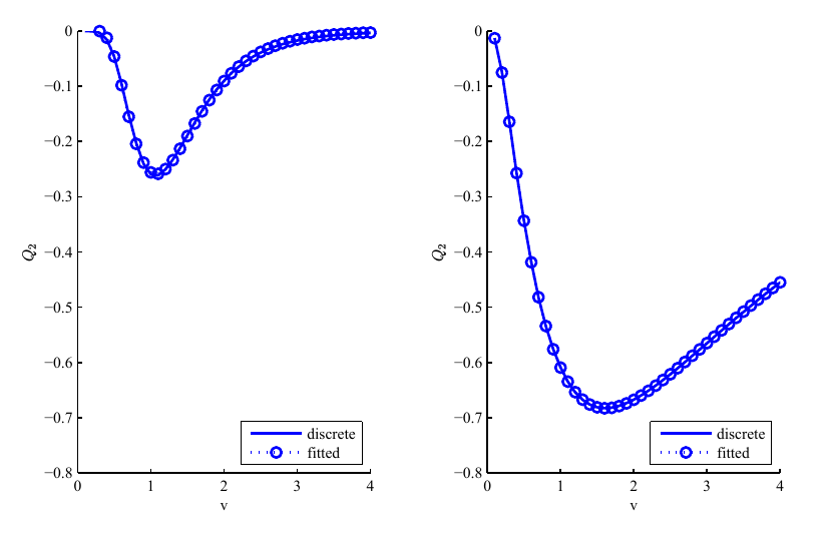}
		\caption{The discrete and fitted values of $Q_2(\tilde{S},v,\tau)$ on $v=0.8, \tau=0.4$ (left) and $v=3.6, \tau=1.5$ (right).}
		\label{graph:fit-Ex7}
	\end{figure}

With $h=0.1$ fixed, Figure \ref{graph:V_Sfixed-Ex7} shows the 2nd-order asymptotic solution and the numerical solutions with the original boundary condition by Heston, MApABC1 and MApABC2 on $\tilde{S}=4$ in the numerical experiments setting $\tilde{S}_{max}=4$ and $\tilde{S}_{max}=6$. Correspondingly, the errors of the numerical solutions with the three different ABCs w.r.t. the reference solution are reported in Figure \ref{graph:Error_Sfixed-Ex7}. As $v$ increases, the errors of  the results with the three ABCs grow larger in general. When we choose $\tilde{S}_{max}=4$, the solution of the experiment using MApABC2 is the closest to the reference solution, MApABC1 is the second and the original boundary condition by Heston behaves the worst. For a greater $\tilde{S}_{max}$, the errors of the experiments using the three boundary conditions all decrease significantly and MApABC2 remains the best.

%Smax=4, 6
	\begin{figure}[htbp!]
		\centering
			\includegraphics[width=0.9\linewidth]{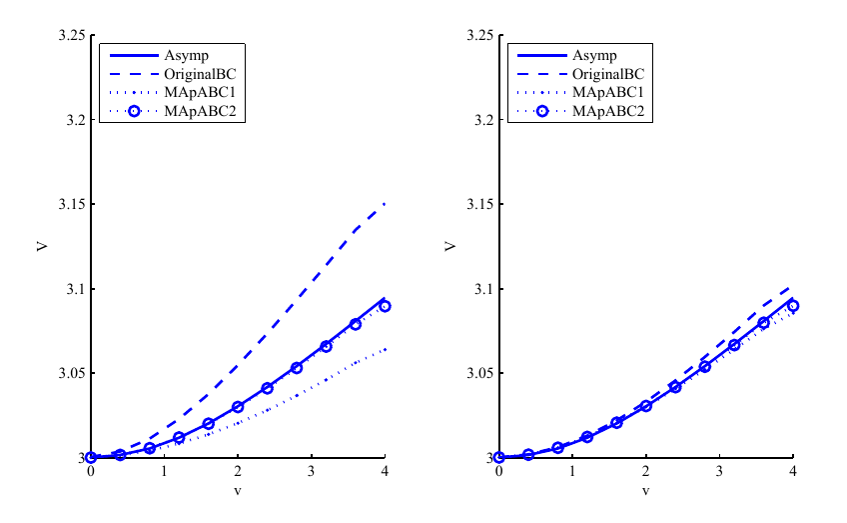}
		\caption{The 2nd-order asymptotic solution and the numerical solutions with different ABCs on $\tilde{S}=4$ with $\tilde{S}_{max}=4$ (left) and $\tilde{S}_{max}=6$ (right).}
		\label{graph:V_Sfixed-Ex7}
	\end{figure}
	\begin{figure}[htbp!]
		\centering
			\includegraphics[width=0.9\linewidth]{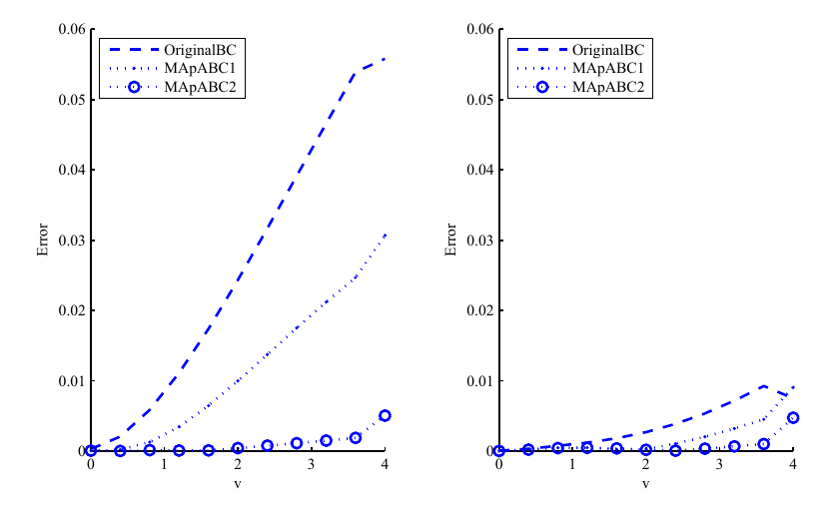}
		\caption{The errors of the numerical solutions with different ABCs w.r.t the 2nd-order asymptotic solution on $\tilde{S}=4$ with $\tilde{S}_{max}=4$ (left) and $\tilde{S}_{max}=6$ (right).}
		\label{graph:Error_Sfixed-Ex7}
	\end{figure}
	
We also choose $v_{max}=4$ and $v_{max}=6$ and show the 2nd-order asymptotic solution and the numerical solutions with the three boundary conditions on $v=4$ in Figure \ref{graph:V_vfixed-Ex7}. The corresponding errors are reported in Figure \ref{graph:Error_vfixed-Ex7} as well. The four solutions seem close to each other for both $v_{max}=4$ and $v_{max}=6$. Still, MApABC2 has a clear advantage over the other two in terms of the error. The errors have almost no change when $v_{max}$ is increased from 4 to 6, which indicates that it is enough to set $v_{max}=4$ on the v-direction to truncate the domain and use the boundary condition.
	
%vmax =4, 6
	\begin{figure}[htbp!]
		\centering
			\includegraphics[width=0.9\linewidth]{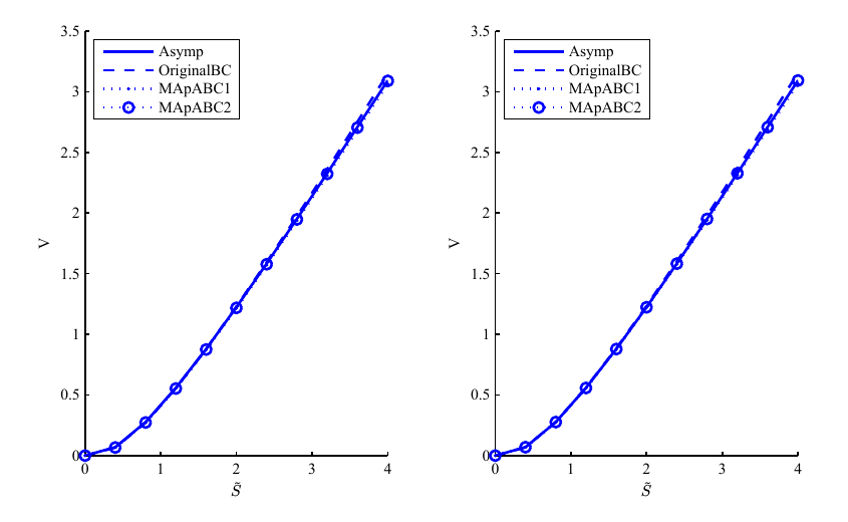}
		\caption{The 2nd-order asymptotic solution and the numerical solutions with different ABCs on $v=4$ with $v_{max}=4$ (left) and $v_{max}=6$ (right).}
		\label{graph:V_vfixed-Ex7}
	\end{figure}

	\begin{figure}[htbp!]
		\centering
			\includegraphics[width=0.9\linewidth]{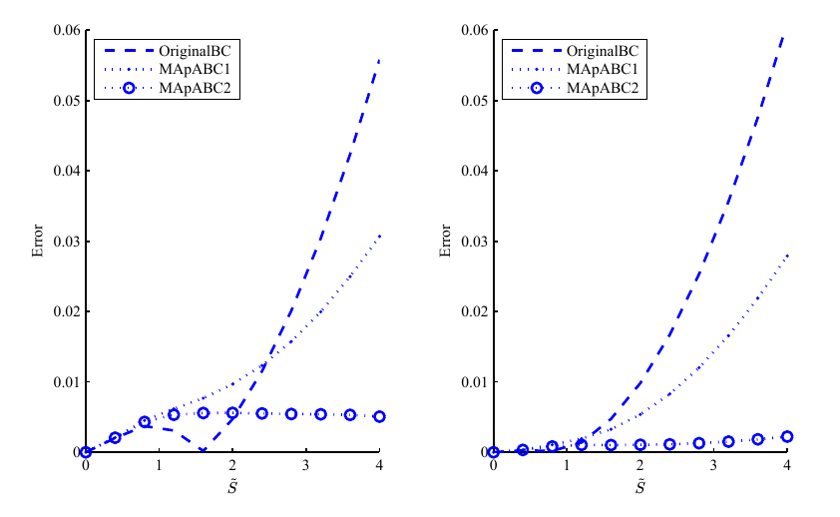}
		\caption{The errors of the numerical solutions with different ABCs w.r.t the 2nd-order asymptotic solution on $v=4$ with $v_{max}=4$ (left) and $v_{max}=6$ (right).}
		\label{graph:Error_vfixed-Ex7}
	\end{figure}
	
%--%--%--%--%--%--%--%--%--%--%--%--%--%--%--%--%--%
Table \ref{tabl:Parameters-Ex6} shows the second set of parameters.
%Parameters
	\begin{table}[ht]
		\centering
		\caption{Parameters}
		\label{tabl:Parameters-Ex6}
		\begin{tabular}{cr}
			\hline
			Parameter & Value \\
			\hline
			$\kappa$ & 0.005\\
			$\eta$ & 0.5\\
			$\sigma$ & 0.01\\
			$\rho$ & 0.5\\
			$T$ & 2\\
			\hline
		\end{tabular}
	\end{table}

Similarly, we choose $S_{max}=4$ and $v_{max}=4$ and present the relative errors using the three different ABCs with different steps in Table \ref{tabl:Errors-Ex6}. For $h=0.4$, MApABC1 and MApABC2 reduce the error both by about 81\% compared to the original boundary condition by Heston while the error of MApABC2 is slightly larger than MApABC1. With the step $h$ getting smaller, the error of MApABC2 decreases faster than the error of MApABC1. For $h=0.025$, the reduction ratios of the error both grow to nearly 99\%. In this case, the improvement by using MApABC1 or MApABC2 is significant.

%Error
	\begin{table}[ht]
		\centering
		\caption{Relative errors}
		\label{tabl:Errors-Ex6}
		\begin{tabular}{crrrrrr}
			\hline
			Steps & $h=0.4$ & $h=0.2$ & $h=0.1$ & $h=0.05$ & $h=0.025$\\
			\hline
			OriginalBC & 0.04037 & 0.03765 & 0.03656 & 0.03616 & 0.03600\\
			MApABC1 & 0.00784 & 0.00276 & 0.00096 & 0.00061 & 0.00052\\
			MApABC2 & 0.00787 & 0.00281 & 0.00097 & 0.00044 & 0.00037\\
			\hline
		\end{tabular}
	\end{table}

With $h=0.1$ fixed, we try to expand the region by choosing larger $\tilde{S}_{max}$ and $v_{max}$ and show the relative errors in Table \ref{tabl:Cost-Ex6}. When we set $\tilde{S}_{max}$ 10 times larger, the relative errors with the three ABCs all get smaller. However, the error of the solution applying the original boundary condition by Heston is still larger than the error of those using MApABC1 and MApABC2 but calculated in the small region, which indicates that the ABCs we propose can help to reduce the computational cost considerably by computing only in a small region while being able to meet the accuracy requirements. With $v_{max}$ 10 times larger, the error corresponding to the original boundary condition has almost no change but grows larger slightly. The reason should be that the imprecision caused by the boundary condition on $\tilde{S}_{max}$ is not improved by a larger $v_{max}$.

%Cost
	\begin{table}[ht]
		\centering
		\caption{Computational cost}
		\label{tabl:Cost-Ex6}
		\begin{tabular}{crrrrrrrr}
			\hline
			Domains & & $S_{max}=4$ & $S_{max}=40$ & $S_{max}=4$\\
               					& & $v_{max}=4$ & $v_{max}=4$ & $v_{max}=40$\\
			\hline
			OriginalBC & Error & 0.03656 & 0.00099  & 0.03665\\
			MApABC1 & Error &  0.00096 & 0.00077 & 0.00074\\
			MApABC2 & Error &  0.00097 & 0.00077 & 0.00078\\
 			\hline
		\end{tabular}
	\end{table}

%--%--%--%--%--%--%--%--%--%--%--%--%--%--%--%--%--%

The third set of parameters are shown in Table \ref{tabl:Parameters-Ex8}.
%Parameters
	\begin{table}[ht]
		\centering
		\caption{Parameters}
\label{tabl:Parameters-Ex8}
\begin{tabular}{cr}
\hline
Parameter & Value \\
\hline
$\kappa$ & 2\\
$\eta$ & 0.3\\
$\sigma$ & 0.05\\
$\rho$ & 0\\
$T$ & 2\\
\hline
\end{tabular}
\end{table}
Here, we choose $S_{max}=8$, $v_{max}=4$. In Table \ref{tabl:Errors-Ex8}, MApABC2 has a much better performance than the results with the original boundary condition by Heston and MApABC1.

%Error
\begin{table}[ht]
\centering
\caption{Relative errors}
\label{tabl:Errors-Ex8}
\begin{tabular}{crrrrrr}
\hline
Steps & $h=0.4$ & $h=0.2$ & $h=0.1$ & $h=0.05$ & $h=0.025$\\
\hline
OriginalBC & 0.00573 & 0.00510& 0.00489 & 0.00481 &  0.00476\\
MApABC1 & 0.00236 & 0.00192 & 0.00185 & 0.00176 & 0.00169\\
MApABC2 & 0.00192 & 0.00090 & 0.00058 & 0.00041 & 0.00030\\
% trunc时最后一个是9.649367941833543e-05
\hline
\end{tabular}
\end{table}

Besides, we present the absolute errors of the numerical solution and also the Greeks including Delta, Gamma and Vega in Figure \ref{graph:ErrorofGreeks-Ex8} corresponding to the three ABCs on the boundary $\tilde{S}_{max}=8$. Greeks have drawn much attention in industry due to their important role in trading and hedging and the computational accuracy is often highly required. We can see easily in the graph that all of the three Greeks are calculated better with the two proposed ABCs especially MApABC2, reflecting the effectiveness and efficiency of our method.

%Greeks comparison
\begin{figure}[htbp!]
	\centering
	\subfigure{
		\includegraphics[width=0.9\linewidth]{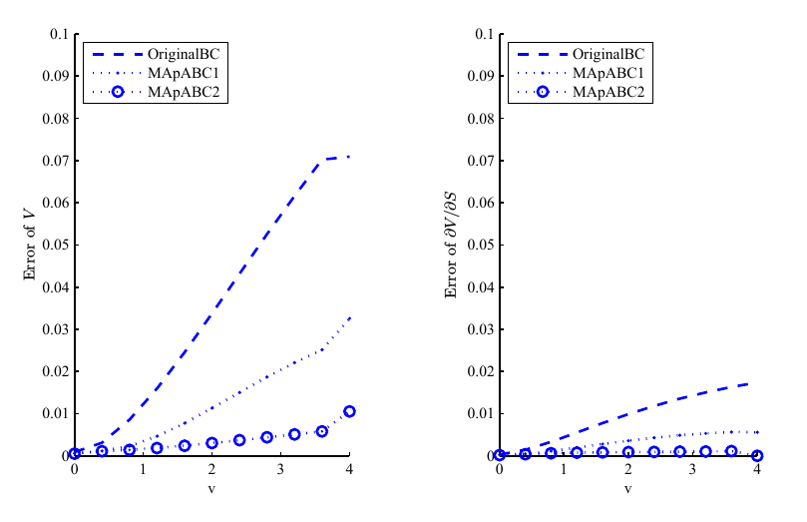}
	}
	\subfigure{
		\includegraphics[width=0.9\linewidth]{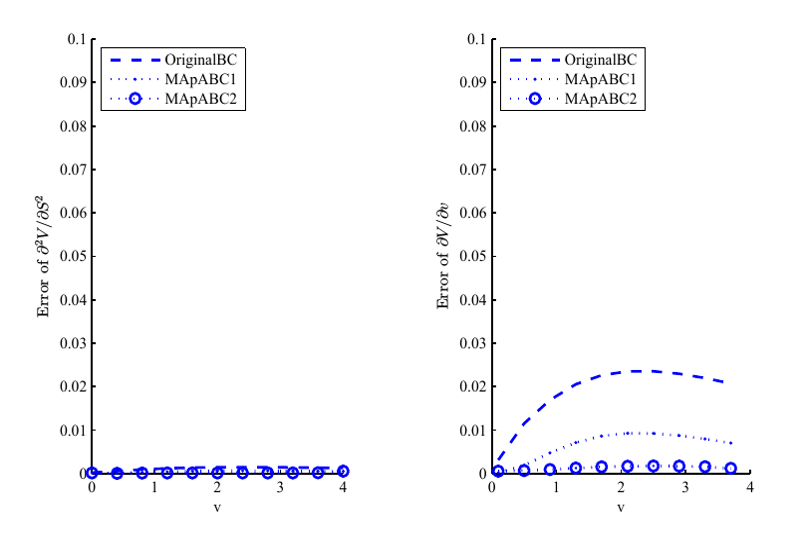}
	}
	\caption{The errors of the numerical solutions with different ABCs, Delta, Gamma and Vega w.r.t the 2nd-order asymptotic solution on $\tilde{S}=8$.}
	\label{graph:ErrorofGreeks-Ex8}
\end{figure}

%--%--%--%--%--%--%--%--%--%--%--%--%--%--%--%--%--%

\section{Conclusions}\label{sec:con}

We derive the asymptotic solution to the problem for pricing the European option under the Heston model in terms of the volatility of variance $\sigma$ and studied the asymptotic behaviors of the derivatives of the price. Based on these asymptotic analyses, we build several artificial boundary conditions for the problem. Numerical study has shown that these ABCs have a better performance than the original boundary condition by Heston and can achieve a good accuracy on a small computational domain for both the option price and the Greeks. A more efficient approach to calculate the integral in the ABCs and boundary conditions on the $v$-direction will be investigated further in the future.

\bibliographystyle{siam}
\bibliography{ref}

\end{document}